\documentclass[a4paper,11pt]{amsart}
\usepackage[latin1]{inputenc}
\usepackage[english]{babel}
\usepackage{amsmath, amsthm, amssymb, amsopn, amsfonts, amstext, stmaryrd, enumerate, color, xcolor, mathtools, mathrsfs, graphicx, chngcntr}
\usepackage[final]{showkeys}
\usepackage{hyperref}
\numberwithin{equation}{section}

\definecolor{RED}{HTML}{F0051C}
\definecolor{BLUE}{HTML}{0629B3}

\hypersetup{
	colorlinks=true,
	linkcolor=BLUE,
	citecolor=RED,
	urlcolor=BLUE,
	linktoc=all}

\DeclareGraphicsExtensions{.png}
\DeclareGraphicsExtensions{.pdf}

\DeclareMathOperator{\supp}{supp}
\DeclareMathOperator*{\osc}{osc}

\newcommand{\C}{\mathscr{C}}
\newcommand{\D}{\mathcal{D}}
\newcommand{\E}{\mathscr{E}}

\newcommand{\K}{\mathscr{K}}

\newcommand{\N}{\mathbb{N}}
\renewcommand{\P}{\mathscr{P}}
\newcommand{\R}{\mathbb{R}}

\newcommand{\Q}{\mathbb{Q}}
\newcommand{\Z}{\mathbb{Z}}

\newcommand{\loc}{{\rm loc}}

\newcommand{\dist}{{\mbox{\normalfont dist}}}
\newcommand{\tRn}{{\widetilde{\R}^n}}
\renewcommand{\S}{{\mathcal{S}_\omega^M}}

\renewcommand{\u}{{u_\omega^M}}
\newcommand{\F}{{\mathscr{F}_\omega}}
\newcommand{\A}{{\mathcal{A}_\omega^M}}
\newcommand{\M}{{\mathcal{M}_\omega^M}}
\newcommand{\tRnm}{{\widetilde{\R}_m^n}}
\newcommand{\um}{{u_{\omega, m}^{A, B}}}
\newcommand{\Fm}{{\mathscr{F}_{\omega, m}}}
\newcommand{\Am}{{\mathcal{A}_{\omega, m}^{A, B}}}
\newcommand{\Mm}{{\mathcal{M}_{\omega, m}^{A, B}}}
\newcommand{\hatum}{{\hat{u}_{\omega, m}^{A, B}}}
\newcommand{\tSm}{{\widetilde{\mathcal{S}}_{\omega, m}^{A, B}}}
\newcommand{\SAB}{{\mathcal{S}_\omega^{A, B}}}

\newcommand{\uAB}{{u_\omega^{A, B}}}
\newcommand{\AAB}{{\mathcal{A}_\omega^{A, B}}}
\newcommand{\MAB}{{\mathcal{M}_\omega^{A, B}}}

\newcommand{\Haus}{\mathcal{H}}

\newcommand{\RR}{\mathscr{R}}

\def\Xint#1{\mathchoice
{\XXint\displaystyle\textstyle{#1}}%
{\XXint\textstyle\scriptstyle{#1}}%
{\XXint\scriptstyle\scriptscriptstyle{#1}}%
{\XXint\scriptscriptstyle\scriptscriptstyle{#1}}%
\!\int}
\def\XXint#1#2#3{{\setbox0=\hbox{$#1{#2#3}{\int}$ }
\vcenter{\hbox{$#2#3$ }}\kern-.6\wd0}}

\def\dashint{\Xint-}

\theoremstyle{plain}
\newtheorem{definition}{Definition}[section]
\newtheorem{theorem}[definition]{Theorem}
\newtheorem{proposition}[definition]{Proposition}
\newtheorem{lemma}[definition]{Lemma}
\newtheorem{corollary}[definition]{Corollary}

\theoremstyle{definition}
\newtheorem{remark}[definition]{Remark}

\renewcommand{\le}{\leqslant}

\renewcommand{\ge}{\geqslant}

\begin{document}

\title[Non-local plane-like minimizers in a periodic medium]{Plane-like minimizers for a non-local Ginzburg-Landau-type energy in a periodic medium}

\author[Matteo Cozzi, Enrico Valdinoci]{
Matteo Cozzi${}^{(1,2)}$
\and
Enrico Valdinoci${}^{(3,4,5)}$
}



\thanks{The work described in this paper has been supported by the ERC grant 277749 ({\it $\varepsilon$~Elliptic Pde's and
Symmetry of Interfaces and Layers for Odd Nonlinearities}) and the PRIN grant
201274FYK7 ({\it Critical Point Theory
and Perturbative Methods for Nonlinear Differential Equations}). The first author is also supported by the~\emph{Mar\'ia de Maeztu} MINECO grant~MDM-2014-0445 and the MINECO grant~MTM2014-52402-C3-1-P}

\maketitle

\date{}

{
\scriptsize \begin{center}
(1) -- BGSMath Barcelona Graduate School of Mathematics.
\end{center}
\scriptsize \begin{center}
(2) -- Departament de Matem\`atiques\\
Universitat Polit\`ecnica de Catalunya\\
Diagonal 647, E-08028 Barcelona (Spain).
\end{center}
\scriptsize \begin{center}
(3) -- Weierstra{\ss} Institut f\"ur Angewandte Analysis und Stochastik\\
Mohrenstra{\ss}e 39, D-10117 Berlin (Germany).
\end{center}
\scriptsize \begin{center}
(4) -- Dipartimento di Matematica ``Federigo Enriques''\\
Universit\`a degli Studi di Milano,\\
Via Saldini 50, I-20133 Milano (Italy).
\end{center}
\scriptsize \begin{center}
(5) -- School of Mathematics and Statistics\\
University of Melbourne\\
Grattan Street, Parkville, VIC-3010 Melbourne (Australia).
\end{center}
\medskip
\begin{center}
E-mail addresses: matteo.cozzi@upc.edu,
enrico.valdinoci@wias-berlin.de
\end{center}
}
\bigskip

\begin{abstract}
We consider a non-local phase transition equation set in a periodic medium and we construct solutions whose interface stays in a slab
of prescribed direction and universal width. The solutions constructed also enjoy a local minimality property with respect to a suitable non-local energy functional.
\end{abstract}

\bigskip

\tableofcontents

\section{Introduction}

The goal of this paper is to construct
solutions of a scalar, fractional Ginzburg-Landau (or Allen-Cahn)
equation in a periodic medium, whose interface
stays in a prescribed slab and whose energy is minimal
among compact perturbations.

The simplest case that we have in mind is the
non-local equation
\begin{equation}\label{EXAMPLE}
(-\Delta)^s u(x) = Q(x)\,\big(u(x)-u^3(x)\big),\end{equation}
in which~$s\in(0,1)$ is a fractional
parameter and~$Q$ is a smooth function, bounded and bounded away
from zero, and such that
\begin{equation}\label{PER:EXAMPLE}
Q(x+k)=Q(x) \;{\mbox{ for every }}\; k\in\Z^n.\end{equation}
The operator~$(-\Delta)^s$ in~\eqref{EXAMPLE}
is a fractional power of the Laplacian, see e.g.~\cite{S06, DPV12}
for an introduction to this topic.

In the framework of equation~\eqref{EXAMPLE}, the
solution~$u:\R^n\to [-1,1]$
represents a state parameter in a model of phase coexistence
(the two ``pure phases'' being represented by $-1$ and~$+1$).
The presence of a fractional exponent~$s\in(0,1)$ is motivated
by models which try to take into account long-range particle
interactions (as a matter of fact, these
models may produce either a local or non-local tension effect,
depending on the value of~$s$, see~\cite{SV12, SV14}; see
also~\cite{PSV13}
for the
variational analysis 
of the different scales of energy that are involved in the model).
\smallskip

We also recall that equations of this type naturally
occur in other areas
of applied mathematics, such as
the Peierls-Nabarro model for crystal dislocations
when~$s=1/2$, and for generalizations of this model when~$s\in(0,1)$
(see e.g.~\cite{N97, DPV15, DFV14}). Related problems also
arise in models for diffusion of biological species (see
e.g.~\cite{F12}).

\smallskip

The periodicity condition in~\eqref{PER:EXAMPLE}
takes into account a possible geometric (or crystalline) structure
of the medium in which the phase transition takes place.

\smallskip

The level sets of the solution~$u$
have particular physical importance, since they correspond,
at a large scale, to the interface between the two phases of system.
The question that we address in this paper is then
to find solutions of~\eqref{EXAMPLE} whose level sets
lie in any given strip of universal size.
The direction of this strip will be arbitrary and the size of
the strip is bounded independently on the direction.
\smallskip

In addition to this geometric constraint on the level sets
of the solution, we will
also prescribe an energy condition. Namely,
equation~\eqref{EXAMPLE}
is variational. Though the associated energy functional
diverges (i.e. nontrivial solutions have infinite
total energy in the whole of the space), it is possible
to ``localize'' the non-local energy density in any fixed domain of
interest and require that the solution has a minimal
property with respect to any perturbation supported in this domain.

\smallskip

The existence of minimal solutions of phase transition equations
whose level sets are confined in a strip goes back to~\cite{V04},
where equation~\eqref{EXAMPLE} was taken into account for~$s=1$
and it is strictly related to the construction, performed in~\cite{CdlL01},
of minimal surfaces which stay at a bounded distance from a plane
(see also \cite{H32, AB06}). Furthermore, these types of results
may be seen as the analogue in partial differential equations
(or pseudo-differential equations) of the classical Aubry-Mather theory
for dynamical systems, see~\cite{M90} (a more detailed discussion
about the existence literature will follow).

\smallskip

As a matter of fact, we will consider here a more
general equation than~\eqref{EXAMPLE}.
Indeed, we will deal with operators that are
more general than the fractional
Laplacian, which can be also spatially heterogeneous and periodic,
and also with more general forcing terms, which may possess
different growths from the pure phases other than the classical
quadratic growth.

\smallskip

The details of the mathematical framework in which we work are the following.
For~$n \ge 2$, we consider the formal energy functional
\begin{equation} \label{Edef}
\E(u) := \frac{1}{2} \int_{\R^n} \int_{\R^n} \left| u(x) - u(y) \right|^2 K(x, y) \, dx dy + \int_{\R^n} W(x, u(x)) \, dx.
\end{equation}

The term~$K: \R^n \times \R^n \to [0, +\infty]$ is supposed to be a measurable and symmetric function, comparable to the kernel of the fractional Laplacian. That is,
\begin{equation} \tag{K1} \label{Ksimm}
K(x, y) = K(y, x) \quad \mbox{for a.e. } x, y \in \R^n,
\end{equation}
and\footnote{Although slightly more general requirements could be imposed on the growth of~$K$ for large values of~$|x - y|$ - see e.g. hypothesis~(1.3) in~\cite{K09} or~(2.2b) in~\cite{C17a} - we prefer to adopt the more restrictive condition~\eqref{Kbounds} in order to simplify the exposition. Requirements~\eqref{Ksimm} and~\eqref{Kbounds} nonetheless allow for a great variety of space-dependent, possibly truncated kernels. In particular, we stress that no regularity is asked on~$K$.}
\begin{equation} \tag{K2} \label{Kbounds}
\frac{\lambda \, \chi_{[0, 1]}(|x - y|)}{|x - y|^{n + 2 s}} \le K(x, y) \le \frac{\Lambda}{|x - y|^{n + 2 s}} \quad \mbox{for a.e. } x, y \in \R^n,
\end{equation}
for some~$\Lambda \ge \lambda > 0$ and~$s \in (0, 1)$.

\smallskip

The mapping~$W$ is a double-well potential, with zeros in~$-1$ and~$1$. More specifically, we assume~$W: \R^n \times \R \to [0, +\infty)$ to be a bounded measurable function for which
\begin{equation} \tag{W1} \label{Wzeros}
W(x, \pm 1) = 0 \quad \mbox{for a.e. } x \in \R^n,
\end{equation}
and, for any~$\theta \in [0, 1)$,
\begin{equation} \tag{W2} \label{Wgamma}
\inf_{\substack{x \in \R^n \\ |r| \le \theta}} W(x, r) \ge \gamma(\theta),
\end{equation}
where~$\gamma$ is a non-increasing positive function of the interval~$[0, 1)$. Moreover, we require~$W$ to be differentiable in the second component, with partial derivative locally bounded in~$r \in \R$, uniformly in~$x \in \R^n$. Accordingly, we let
\begin{equation} \tag{W3} \label{Wbound}
W(x, r), \, |W_r(x, r)| \le W^* \quad \mbox{for a.e. } x \in \R^n \mbox{ and any } r \in [-1, 1],
\end{equation}
for some~$W^* > 0$.

\smallskip

Since we are interested in modelling a periodic environment, we require both~$K$ and~$W$ to be periodic under integer translations. That is,
\begin{equation} \tag{K3} \label{Kper}
K(x + k, y + k) = K(x, y) \quad \mbox{for a.e. } x, y \in \R^n \mbox{ and any } k \in \Z^n,
\end{equation}
and
\begin{equation} \tag{W4} \label{Wper}
W(x + k, r) = W(x, r) \quad \mbox{for a.e. } x \in \R^n \mbox{ and any } k \in \Z^n,
\end{equation}
for any fixed~$r \in \R$.

\smallskip

The assumptions listed above allow us to comprise a very general class of kernels and potentials.

As possible choices for~$K$, we could indeed think of heterogeneous, isotropic kernels of the type
$$
K(x, y) = \frac{a(x, y)}{|x - y|^{n + 2 s}},
$$
for a measurable~$a: \R^n \times \R^n \to [\lambda, \Lambda]$, or instead consider a translation invariant, but anisotropic~$K$, as given by
$$
K(x, y) = \frac{1}{\| x - y \|^{n + 2 s}},
$$
with~$\| \cdot \|$ a measurable norm in~$\R^n$. Furthermore, one can combine both heterogeneity and anisotropy to obtain, for instance, kernels of the form
$$
K(x, y) = \frac{1}{ \langle A(x, y) (x - y), (x - y) \rangle^{\frac{n + 2 s}{2}} },
$$
where~$A$ is a symmetric, uniformly elliptic~$n \times n$ matrix with bounded entries.

Of course, the functions~$a$ and~$A$ should satisfy appropriate symmetry and periodicity conditions, in order that hypotheses~\eqref{Ksimm} and~\eqref{Kper} could be fulfilled by the resulting~$K$'s. Also, such functions may exhibit a
degenerate behavior when~$x$ and~$y$ are far from each other (compare this with the left-hand side of~\eqref{Kbounds}).

\smallskip

Important examples of admissible potentials~$W$ are given by
\begin{equation} \label{Wexample}
W(x, r) = Q(x) \left| 1 - r^2 \right|^d \quad \mbox{or} \quad W(x, r) = Q(x) \left( 1 + \cos \pi r \right),
\end{equation}
with~$d > 1$ and~$Q$ a positive periodic function.\footnote{When comparing these assumptions with those usually found in the related literature on local functionals, see e.g.~\cite{CC95,CC06} or~\cite{V04}, one realizes that the parameter~$d$ is asked there to range in the interval~$(0, 2]$. This is due
essentially to the fact that our proofs do not rely on the density estimates established in those papers, but on some H\"{o}lder regularity results.

If on the one hand this enables us to consider extremely flat potentials near the zeroes~$-1$ and~$1$, which can be obtained by taking~$d > 2$, on the other hand the Lipschitz continuity needed on~$W$ for the regularity results to apply imposes the bound~$d > 1$. This is due to the fact that our regularity theory is really designed for solutions to integro-differential equations, instead of minimizers.

Note added in proof: see Section~\ref{AiPsec} for a discussion around the possibility of circumventing this issue and considering the whole array of exponents~$d > 0$.

} 
By taking~$W(x,r):=Q(x)(1-r^2)^2$
and~$K(x,y):=|x-y|^{-n-2s}$, one obtains that
the critical points of the energy functional
satisfy the model equation in~\eqref{EXAMPLE}
(up to normalization constants).
\medskip

In the present work we look for minimizers of the functional~$\E$ which connects the two pure phases~$-1$ and~$1$, which are
the zeroes of the potential~$W$. In particular, given any vector~$\omega \in \R^n \setminus \{ 0 \}$, we  address the existence of minimizers for which, roughly speaking, \emph{most} of the transition between the pure states occurs in a strip orthogonal to~$\omega$ and of universal width. Moreover, when~$\omega$ is a rational vector, we want our minimizers to exhibit some kind of periodic behavior, consistent with that of the ambient space.

\smallskip

Note that we will often call a quantity \emph{universal} if it depends at most on~$n$,~$s$,~$\lambda$,~$\Lambda$,~$W^*$ and on the function~$\gamma$ introduced in~\eqref{Wgamma}.

\medskip

In order to formulate an exact statement, we introduce the following terminology. For a given~$\omega \in \Q^n \setminus \{ 0 \}$, we consider in~$\R^n$ the relation~$\sim_\omega$ defined by setting
\begin{equation} \label{simrel}
x \sim_\omega y \quad \mbox{if and only if} \quad y - x = k \in \Z^n, \mbox{ with } \omega \cdot k = 0.
\end{equation}
Notice that~$\sim_\omega$ is an equivalence relation and that the associated quotient space
$$
\widetilde{\mathbb{R}}^n_\omega := \R^n / \sim_\omega,
$$
is topologically the Cartesian product of an~$(n - 1)$-dimensional torus and a line. We say that a function~$u: \R^n \to \R$ is \emph{periodic with respect to~$\sim_\omega$}, or simply~\emph{$\sim_\omega$-periodic}, if~$u$ respects the equivalence relation~$\sim_\omega$, i.e. if
$$
u(x) = u(y) \quad \mbox{for any } x, y \in \R^n \mbox{ such that } x \sim_\omega y.
$$
When no confusion may arise, we will indicate the relation~$\sim_\omega$ just by~$\sim$ and the resulting quotient space by~$\tRn$.

\smallskip

To specify the notion of minimizers that we take into consideration,
we need to introduce an appropriate localized energy functional.
Given a set~$\Omega \subseteq \R^n$ and a function~$u: \R^n \to \R$, we define the \emph{total energy}~$\E$ of~$u$ in~$\Omega$ as
\begin{equation} \label{EOmegadef}
\E(u; \Omega) := \frac{1}{2} \iint_{\C_\Omega} \left| u(x) - u(y) \right|^2 K(x, y) \, dx dy + \int_{\Omega} W(x, u(x)) \, dx,
\end{equation}
where
\begin{equation} \label{COmega}
\begin{aligned}
\C_\Omega := & \left( \R^n \times \R^n \right) \setminus \left( \left( \R^n \setminus \Omega \right) \times \left( \R^n \setminus \Omega \right) \right) \\
= & \left( \Omega \times \Omega \right) \cup \left( \Omega \times \left( \R^n \setminus \Omega \right) \right) \cup \left( \left( \R^n \setminus \Omega \right) \times \Omega \right).
\end{aligned}
\end{equation}
Notice that when~$\Omega$ is the whole space~$\R^n$, then the energy~\eqref{EOmegadef} coincides with that anticipated in~\eqref{Edef}.

Sometimes, a more flexible notation for this functional will turn out to be useful. To this aim, recalling our symmetry assumption~\eqref{Ksimm} on~$K$, we will refer to~$\E(u; \Omega)$ as the sum of the \emph{kinetic part}\footnote{We stress that the name~\emph{kinetic} does not hint at actual physical motivations. In fact, in the applications~$\K$ is typically used to describe non-local interactions and elastic forces. However, we adopt this slight abuse of terminology in conformity with the classical jargon used for local Dirichlet energies
in particle mechanics. It is of course an interesting problem to 
study also
more general types of kinetic energies, such as the ones which lead
to quasilinear fractional equations, having an
integrability growth different than quadratic, see e.g.~\cite{DKV16, BL17} and the references therein.}
$$
\K(u; \Omega, \Omega) + 2 \K(u; \Omega, \R^n \setminus \Omega),
$$
with
$$
\K(u; U, V) := \frac{1}{2} \int_U \int_V |u(x) - u(y)|^2 K(x, y) \, dx dy,
$$
for any~$U, V \subseteq \R^n$, and the \emph{potential part}
$$
\P(u; \Omega) := \int_\Omega W(x, u(x)) \, dx.
$$

With this in hand, the notion of minimization inside a bounded set is described by the following
\begin{definition} \label{mindef}
Let~$\Omega$ be a bounded subset of~$\R^n$. A function~$u$ is said to be a~\emph{local minimizer} of~$\E$ in~$\Omega$ if~$\E(u; \Omega) < +\infty$ and
\begin{equation} \label{mindefeq}
\E(u; \Omega) \le \E(v; \Omega),
\end{equation}
for any~$v$ which coincides with~$u$ in~$\R^n \setminus \Omega$.
\end{definition}

For simplicity, in Definition~\ref{mindef}
and throughout the paper we assume every set and every function to be measurable, even if it is not explicitly stated.

\begin{remark} \label{restrmk}
We point out that a minimizer~$u$ on~$\Omega$ is also a minimizer on every subset of~$\Omega$. Though not obvious, this property is easily justified as follows.

Let~$\Omega' \subset \Omega$ be measurable sets and~$v$ be a function coinciding with~$u$ outside~$\Omega'$. Recalling the notation introduced in~\eqref{COmega}, it is immediate to check that~$\C_{\Omega'} \subset \C_{\Omega}$ and
$$
\C_\Omega \setminus \C_{\Omega'} = \left( \left( \Omega \setminus \Omega' \right) \times \left( \Omega \setminus \Omega' \right) \right) \cup \left( \left( \Omega \setminus \Omega' \right) \times \left( \R^n \setminus \Omega \right) \right) \cup \left( \left( \R^n \setminus \Omega \right) \times \left( \Omega \setminus \Omega' \right) \right).
$$
Therefore, it follows that the integrands of the kinetic parts of~$\E(u; \Omega)$ and~$\E(v; \Omega)$ coincide on~$\C_\Omega \setminus \C_{\Omega'}$. Since also the respective arguments of the potential terms are equal on~$\Omega \setminus \Omega'$, by~\eqref{mindefeq} we conclude that
\begin{align*}
\E(u; \Omega') & = \E(u; \Omega) - \frac{1}{2} \iint_{\C_\Omega \setminus \C_{\Omega'}} |u(x) - u(y)|^2 K(x, y) \, dx dy - \P(u; \Omega \setminus \Omega') \\
& \le \E(v; \Omega) - \frac{1}{2} \iint_{\C_\Omega \setminus \C_{\Omega'}} |v(x) - v(y)|^2 K(x, y) \, dx dy - \P(v; \Omega \setminus \Omega') \\
& = \E(v; \Omega').
\end{align*}
Thus,~$u$ is a minimizer on~$\Omega'$.
\end{remark}

Up to now we only discussed about local minimizers. Since we plan to construct functions which exhibit minimizing properties on the full space, we need to be precise on how we mean to extend Definition~\ref{mindef} to
the whole of~$\R^n$ (where the total
energy functional may be divergent).

\begin{definition} \label{classAdef}
A function~$u$ is said to be a~\emph{class~A minimizer} of the functional~$\E$ if it is a minimizer of~$\E$ in~$\Omega$, for any bounded set~$\Omega \subset \R^n$.
\end{definition}

\medskip

Now that all the main ingredients have been introduced, we are ready to state formally the main result of the paper.

\begin{theorem} \label{mainthm}
Let~$n \ge 2$ and~$s \in (0, 1)$. Assume that the kernel~$K$ and the potential~$W$ satisfy~\eqref{Ksimm},~\eqref{Kbounds},~\eqref{Kper} and~\eqref{Wzeros},~\eqref{Wgamma},~\eqref{Wbound},~\eqref{Wper}, respectively.\\
For any fixed~$\theta \in (0, 1)$, there is a constant~$M_0 > 0$, depending only on~$\theta$ and on universal quantities, such that, given any~$\omega \in \R^n \setminus \{ 0 \}$, there exists a class~A minimizer~$u_\omega$ of the energy~$\E$ for which the level set~$\{ |u_\omega| < \theta \}$ is contained in the strip
$$
\left\{ x \in \R^n : \frac{\omega}{|\omega|} \cdot x \in [0, M_0] \right\}.
$$
Moreover,
\begin{enumerate}[$\bullet$]
\item if~$\omega \in \Q^n \setminus \{ 0 \}$, then~$u_\omega$ is periodic with respect to~$\sim_\omega$, while
\item if~$\omega \in \R^n \setminus \Q^n$, then~$u_\omega$ is the uniform limit on compact subsets of~$\R^n$ of a sequence of periodic class~A minimizers.
\end{enumerate}
\end{theorem}

We remark that Theorem~\ref{mainthm} is new even in the model
case in which~$W(x,r):=Q(x)(1-r^2)^2$
and~$K(x,y):=|x-y|^{-n-2s}$. In this case, Theorem~\ref{mainthm}
provides solutions of equation~\eqref{EXAMPLE} (up to normalizing
constants).\medskip

In the local case - which formally corresponds to taking~$s = 1$ and can be effectively realized by replacing our kinetic term with the Dirichlet-type energy
\begin{equation} \label{localkin}
\int \langle A(x) \nabla u(x), \nabla u(x) \rangle \, dx,
\end{equation}
where~$A$ is a bounded, uniformly elliptic matrix - the result contained in Theorem~\ref{mainthm} was proved by the second author in~\cite{V04}. After this, several generalizations were obtained, extending such result in many directions. See, for instance,~\cite{PV05,NV07,dlLV07,BV08} and~\cite{D13}. We also mention the pioneering work~\cite{CdlL01} of Caffarelli and de la Llave, where the two authors proved the existence of plane-like minimal surfaces with respect to periodic metrics of~$\R^n$.

We stress that, if we restrict to the case given by~$K(x, y) := (1 - s) |x - y|^{- n - 2s}$, it can be proved that Theorem~\ref{mainthm} is stable as~$s$ approaches~$1$. As a consequence, by taking this limit one may deduce from it~\cite[Theorem~8.1]{V04}, at least for the model case of~$A$ equal to the identity matrix in~\eqref{localkin}. We refer the interested reader to Section~\ref{sto1sec} for a rigorous presentation of these arguments.

\medskip

The proof of Theorem~\ref{mainthm} 
makes use of a geometric and variational
technique developed in~\cite{CdlL01} and~\cite{V04}, suitably adapted
in order to deal with non-local interactions.
For a given rational direction~$\omega \in \Q^n \setminus \{ 0 \}$ and a fixed strip
$$
\S := \left\{ x \in \R^n : \omega \cdot x \in [0, M] \right\},
$$
with~$M > 0$, one takes advantage of the identifications of the quotient space~$\tRn$ to gain the compactness needed to obtain a minimizer~$\u$ with respect to periodic perturbations supported inside~$\S$. By construction, this minimizer is such that its interface~$\{|\u| < \theta \}$ is contained in the strip~$\S$.

With the aid of some geometrical arguments, one then shows that~$\u$ becomes a class~A minimizer for~$\E$, provided~$M / |\omega|$ is larger than some universal parameter~$M_0$. The fact that the threshold~$M_0$ is universal and that, in particular, it does not depend on the fixed direction~$\omega$ is of key importance here
and it allows, as a byproduct,
to obtain the result for an irrational
vector~$\omega \in \R^n \setminus \Q^n$, by taking
the limit of rational directions.
\medskip

We remark that the non-local character of the energy~$\E$ introduces several
challenging difficulties into the above scheme.

\smallskip

First of all, the way the compactness is used to construct the minimizer~$\u$ is somehow not as straightforward as in the local case.

To have a glimpse of this difference, consider that in~\cite{V04} the candidate~$\u$ is by definition a minimizer with respect to~$\sim$-periodic perturbations occurring in~$\S$. That is, one really considers the energy~$\E$ driven by~\eqref{localkin} as defined on the cylinder~$\tRn$ viewed as a manifold and obtain~$\u$ as the absolute minimizer of~$\E$ within a particular class of functions defined on~$\tRn$. However, since the restriction of the local kinetic term~\eqref{localkin} to a fundamental domain of~$\tRn$ only~\emph{sees} what happens inside that domain, it is clear that one is allowed in the local case
to identify periodic perturbations and perturbations which are compactly supported inside~$\tRn$. As a result,~$\u$ is automatically a local minimizer for~$\E$ in the strip~$\S$.

As it is, this technique cannot work in the
non-local setting. Indeed, let~$u$ be any~$\sim$-periodic function and~$\varphi$ be compactly supported in a fixed fundamental region~$D$ of~$\tRn$: if we denote by~$\tilde{\varphi}$ the~$\sim$-periodic extension of~$\varphi|_D$ to~$\R^n$, then the two quantities~$\E(u + \varphi; D)$ and~$\E(u + \tilde{\varphi}; D)$, as defined in~\eqref{EOmegadef}, are not equal in general.

In order to overcome this difficulty, we introduce an appropriate auxiliary functional~$\F$ that is used to define the periodic minimizer~$\u$. Then, it happens that~$\u$ is a local minimizer for the original energy~$\E$, since~$\F$ couples with~$\E$ in a favorable way.

\smallskip

An additional
difficulty comes from the different asymptotic properties
of the energy in terms of the fractional parameter~$s$.
As a matter of fact, the threshold~$s=1/2$
distinguishes the local and non-local behavior of the functional
at a large scale (see~\cite{SV12, SV14}) and it reflects
into the finiteness or infiniteness of the energy of
the one-dimensional transition layer.
In our setting, this feature implies that
not all the kernels~$K$ satisfying~\eqref{Kbounds} can be dealt with at the same time. More precisely, when~$s \le 1 / 2$ the behavior at infinity dictated by~\eqref{Kbounds} causes
infinite contributions coming from far. For this reason, at least at a first glance,
it may seem necessary
to restrict the class of admissible kernels by imposing some additional requirements on the decay of~$K$ at infinity.
However, we will be able to remove 
this limitation by an appropriate limit procedure.
Namely, we will first assume a fast decay property
of the kernel to obtain the existence of a class~A minimizer,
but the estimates obtained will be independent of this additional
assumption.
Consequently, we will be able to extend the result to general kernels by treating them as limits of truncated ones.

\smallskip

Finally, we want to point out a possibly interesting difference between the proof displayed here and that of e.g.~\cite{CdlL01} and~\cite{V04}. In the existing literature, the technique that is typically adopted to show that~$\u$ is a class~A minimizer relies on the so-called energy and density estimates.

These estimates respectively deal with the growth of the energy~$\E$ of a local minimizer~$u$ inside \emph{large} balls and the fractions of such balls occupied by a fixed level set of~$u$. The latter, in particular, is a powerful tool first introduced by Caffarelli and C\'ordoba in~\cite{CC95} to study the uniform convergence of the level sets of a family of \emph{scaled} minimizers.

Although such density estimates have been established in~\cite{SV14} in a non-local setting very close to ours, for some technical reasons we decided not to incorporate them into our argument
(roughly speaking, the periodic setting is not immediately
compatible with large balls in Euclidean spaces). In their place, we take advantage of some~$C^{0, \alpha}$ bounds satisfied by local minimizers of~$\E$, along with a suitable version of the energy estimates.

The above mentioned H\"{o}lder continuity result is essentially the regularity theory for bounded weak solutions to integro-differential equations developed by Kassmann in~\cite{K09,K11}. On the other hand, energy estimates for minimizers of non-local energies have been independently obtained in~\cite{CC14} and~\cite{SV14} (in different settings). Since both these two results were set in a slightly different framework than ours, we provide their proofs in full details in Sections~\ref{regsec} and~\ref{enestsec}, respectively.

\medskip

The paper is organized as follows.
Sections~\ref{regsec} and~\ref{enestsec} are devoted to the H\"{o}lder regularity of the minimizers and the energy estimates. We stress that in these two sections both~$K$ and~$W$ are subjected to slightly more general requirements than those listed in the introduction
(the statements of the results proved in these sections will contain the precise hypotheses needed for their proofs).

Section~\ref{s>1/2sec} is occupied by the main construction leading to the proof of Theorem~\ref{mainthm}. For the reader's ease, this section is in turn divided into seven short subsections. 
In each of these subsections, we will consider, respectively:
\begin{itemize}
\item the minimization arguments by compactness,
\item the notion of minimal minimizer (i.e. the pointwise infimum
of all the possible minimizers, which satisfy additional
geometric and functional features),
\item the doubling property (roughly, doubling the period does not
change the minimal minimizer), 
\item the notion of minimization
under compact perturbations, 
\item the Birkhoff property
(namely, the level sets of the minimal minimizers are ordered
by integer translations),
\item the passage from constrained to unconstrained minimization
(for large strips, we show that the constraint is irrelevant),
\item the passage from rational to irrational slopes.
\end{itemize}
The argument displayed 
in Section~\ref{s>1/2sec}
only works under an additional assumption on the 
decay rate of the kernel~$K$ at infinity. In the subsequent 
Section~\ref{s<1/2sec} we will show that this hypothesis can be in fact 
removed by a limit procedure. The proof of Theorem~\ref{mainthm} will therefore be completed.

In Section~\ref{sto1sec} we discuss about the stability of Theorem~\ref{mainthm} in some particular cases, as the fractional order~$s$ of the kinetic term goes to~$1$.

We conclude this paper
with two appendices which contain some auxiliary material
needed for the technical steps in the proofs of our main results.

\section{Regularity of the minimizers} \label{regsec}

In this introductory section we show that the local minimizers of~$\E$ are H\"{o}lder continuous functions. In order to do this, we prove a general regularity result for bounded solutions to non-local equations driven by measurable kernels comparable to that of the fractional Laplacian.

In this regard, we stress that the main result of this section - namely, Theorem~\ref{holderthm} - is stated in a broader setting, in comparison with the rest of the paper. The periodicity of the medium does not play any role here and it is therefore not assumed. 

We point out that, while we do not obtain uniform estimates as~$s \rightarrow 1^-$, our result is still independent of~$s$, as long as~$s$ is far from~$0$ and~$1$.

\medskip

Let~$0 < s < 1$ and~$\Omega$ be a bounded open set of~$\R^n$. Let~$K$ be a measurable kernel satisfying~\eqref{Ksimm} and~\eqref{Kbounds}. We now introduce the space of solutions~$X(\Omega)$. Given a measurable function~$u: \R^n \to \R$, we say that~$u \in X(\Omega)$ if and only if
$$
u|_\Omega \in L^2(\Omega) \quad \mbox{and} \quad (x, y) \longmapsto \left( u(x) - u(y) \right) \sqrt{K(x, y)} \in L^2(\C_\Omega).
$$
It is not difficult to see that~\eqref{Kbounds} implies that~$H^s(\R^n) \subset X(\Omega) \subseteq H^s(\Omega)$. We also denote by~$X_0(\Omega)$ the subspace of~$X(\Omega)$ made up by the functions which vanish a.e. outside~$\Omega$. Then~$X_0(\Omega') \subseteq X_0(\Omega) \subset H^s(\R^n)$, if~$\Omega' \subseteq \Omega$. We refer the reader to~\cite[Section~5]{SerV13}, where some useful properties of these spaces are discussed.

We consider the non-local Dirichlet form
\begin{equation} \label{DKdef}
\D_K(u, \varphi) = \int_{\R^n} \int_{\R^n} \left( u(x) - u(y) \right) \left( \varphi(x) - \varphi(y) \right) K(x, y) \, dx dy.
\end{equation}
Observe that~$\D_K$ is well-defined for instance when~$u \in X(\Omega)$ and~$\varphi \in X_0(\Omega)$.

Let now~$f \in L^2(\Omega)$. We say that~$u \in X(\Omega)$ is a \emph{supersolution} of the equation
\begin{equation} \label{ELeq}
\D_K(u, \cdot) = f \quad \mbox{in } \Omega,
\end{equation}
if
\begin{equation} \label{ELsupsol}
\D_K(u, \varphi) \ge \langle f, \varphi \rangle_{L^2(\R^n)} \quad \mbox{for any non-negative } \varphi \in X_0(\Omega).
\end{equation}
Analogously, one defines \emph{subsolutions} of~\eqref{ELeq} by reverting the inequality in~\eqref{ELsupsol} and, as well, \emph{solutions} by asking~\eqref{ELsupsol} to be an identity and neglecting the sign assumption on~$\varphi$. It is almost immediate to check that a function~$u$ is a solution of~\eqref{ELeq} if and only if it is at the same time a super- and a subsolution.

\medskip

The main result of the section is given by the following

\begin{theorem} \label{holderthm}
Let~$\Omega$ be a bounded open set of~$\R^n$, with~$n \ge 2$, and~$s_0 \in (0, 1/2)$ be a fixed parameter. Let~$s \in [s_0, 1 - s_0]$ and~$K$ be a measurable kernel satisfying~\eqref{Ksimm} and~\eqref{Kbounds}. If~$f \in L^\infty(\Omega)$ and~$u \in X(\Omega) \cap L^\infty(\R^n)$ is a solution of~\eqref{ELeq} in~$\Omega$, then there exists an exponent~$\alpha \in (0, 1)$, only depending on~$n$,~$s_0$,~$\lambda$ and~$\Lambda$, such that
$$
u \in C^{0, \alpha}_\loc(\Omega).
$$
In particular, there exists a number~$R_0 > 0$, depending only on~$n$,~$s_0$,~$\lambda$ and~$\Lambda$, such that, for any point~$x_0 \in \Omega$ and any radius~$0 < R \le R_0$ for which~$B_R(x_0) \subset \Omega$, it holds
\begin{equation} \label{holderine}
\osc_{B_r(x_0)} u \le 16 \left( \frac{r}{R} \right)^\alpha \Big[ \| u \|_{L^\infty(\R^n)} + \| f \|_{L^\infty(B_R(x_0))} \Big],
\end{equation}
for any~$0 < r < R$.
\end{theorem}

Theorem~\ref{holderthm} is an extension to non-local equations of the classical De Giorgi-Nash-Moser regularity theory. In recent years a great number of papers dealt with interior H\"{o}lder estimates for solutions of elliptic integro-differential equations, as for instance~\cite{S06,CS09,K09} and~\cite{K11}. See also the recent~\cite{DK15}, which contains related and very interesting regularity results, especially for the case of homogeneous equations. In our setting, we need estimates for equations with general right-hand sides, which apparently are not formally stated nor proved in the literature (although they can be deduced using the techniques of e.g.~\cite{K09} and~\cite{DK15}). Following the arguments of these papers, we provide here below a fully detailed and self-contained proof of these estimates.

Before advancing to the arguments that lead to Theorem~\ref{holderthm}, we point out how the regularity of the minimizers of~$\E$ can be recovered from it.

\begin{corollary} \label{minminholdercor}
Fix~$s_0 \in (0, 1/2)$ and let~$s \in [s_0, 1 - s_0]$. Let~$u$ be a bounded local minimizer of~$\E$ in a bounded open subset~$\Omega$ of~$\R^n$. Then,~$u \in C_\loc^{0, \alpha}(\Omega)$, for some~$\alpha \in (0, 1)$. The exponent~$\alpha$ only depends on~$n$,~$s_0$,~$\lambda$ and~$\Lambda$, while the~$C^{0, \alpha}$ norm of~$u$ on any~$\Omega' \subset \subset \Omega$ may also depend on~$\| u \|_{L^\infty(\R^n)}$,~$\| W_r(\cdot, u) \|_{L^\infty(\Omega)}$ and~$\dist \left( \Omega', \partial \Omega \right)$.
\end{corollary}
\begin{proof}
Let~$u$ be a bounded local minimizer of~$\E$ in~$\Omega$. By taking the first variation of~$\E$, it is easy to see that~$u$ is a solution of the Euler-Lagrange equation~\eqref{ELeq} in~$\Omega$, with~$f = W_r(\cdot, u)$. Notice that~$u \in X(\Omega)$, since~$\E(u; \Omega)$ is finite. Moreover, being~$u \in L^\infty(\R^n)$ and~$W_r$ locally bounded, we obtain that~$f$ is also a bounded function in~$\Omega$. Thence, Theorem~\ref{holderthm} applies and yields the~$C^{0, \alpha}$ regularity of~$u$. The quantitative estimate of the H\"{o}lder norm of~$u$ on compact subsets of~$\Omega$ follows by applying~\eqref{holderine} along with a standard covering argument.
\end{proof}

The remaining part of the section is devoted to the proof of Theorem~\ref{holderthm}, which is based on the
Moser's iteration technique 
and some arguments in~\cite{K09,K11}. 

We begin with a lemma dealing with
non-negative supersolutions of~\eqref{ELeq}.

\begin{lemma} \label{JNlem}
Let~$f \in L^\infty(B_1)$ and~$u \in X(B_1)$ be a non-negative supersolution of~\eqref{ELeq} in~$B_1$. Suppose that
\begin{equation} \label{ugef+delta}
u(x) \ge \| f \|_{L^\infty(B_1)} + \delta \quad \mbox{for a.e. } x \in B_1,
\end{equation}
for some~$\delta > 0$. Then,
\begin{equation} \label{up0est}
\left( \dashint_{B_{1 / 2}} u(x)^{p_\star} \, dx \right)^{1 / p_\star} \le C_\star \left( \dashint_{B_{1 / 2}} u(x)^{- p_\star} \, dx \right)^{- 1 / p_\star},
\end{equation}
for some constant~$C_\star > 0$ and exponent~$p_\star \in (0, 1)$ which depend only on~$n$,~$s_0$,~$\lambda$ and~$\Lambda$.
\end{lemma}
\begin{proof}
We plan to show that~$\log u \in BMO(B_{1 / 2})$. To this aim, we claim that there exists a constant~$c_1 > 0$, depending only on~$n$,~$s_0$,~$\lambda$ and~$\Lambda$, such that
\begin{equation} \label{logclaim}
[\log u]_{H^s(B_r(z))} \le c_1 r^{- s + n / 2},
\end{equation}
holds true for any~$z \in B_{1 / 2}$ and~$r > 0$ for which~$B_r(z) \subseteq B_{1 / 2}$.

In order to prove~\eqref{logclaim}, we take a cut-off function~$\eta \in C^\infty_c(\R^n)$ satisfying~$0 \le \eta \le 1$ in~$\R^n$,~$\supp(\eta) = B_{3 r / 2}(z)$,~$\eta = 1$ in~$B_r(z)$ and~$|\nabla \eta| \le 4 r^{-1}$ in~$\R^n$. We test formulation~\eqref{ELsupsol} with~$\varphi := \eta^2 u^{-1}$. Note that~$\varphi \ge 0$ and~$\varphi \in X_0(B_1)$ thanks to the definition of~$\eta$ and condition~\eqref{ugef+delta}. Recalling~\eqref{Ksimm}, inequality~\eqref{ELsupsol} becomes
\begin{equation}\label{8bis}\begin{split}
\int_{B_{3 r / 2}(z)} \frac{f(x) \eta^2(x)}{u(x)} \, dx & \le \int_{B_{2 r}(z)} \int_{B_{2 r}(z)} \left( u(x) - u(y) \right) \left( \frac{\eta^2(x)}{u(x)} - \frac{\eta^2(y)}{u(y)} \right) K(x, y) \, dx dy \\
& \quad + 2 \int_{B_{2 r}(z)} \frac{\eta^2(y)}{u(y)} \left( \int_{\R^n \setminus B_{2 r}(z)} \left( u(y) - u(x) \right) K(x, y) \, dx \right) dy \\
& =: I_1 + 2 I_2.
\end{split}\end{equation}
For any~$x, y \in B_{2 r}(z)$ we compute
\begin{align*}
\left( u(x) - u(y) \right) \left( \frac{\eta^2(x)}{u(x)} - \frac{\eta^2(y)}{u(y)} \right) & = \eta^2(x) + \eta^2(y) - \frac{\eta^2(x) u(y)}{u(x)} - \frac{\eta^2(y) u(x)}{u(y)} \\
& = \left| \eta(x) - \eta(y) \right|^2 - \frac{\left| \eta(y) u(x) - \eta(x) u(y) \right|^2}{u(x) u(y)}.
\end{align*}
Hence, using~\eqref{Kbounds} together with the numerical inequality
$$
\left( \log a - \log b \right)^2 \le \frac{(a - b)^2}{a b},
$$
that holds for any~$a, b > 0$, we get\footnote{Throughout the paper, the symbol~$\alpha_n$ is used to denote the volume of the unit ball of~$\R^n$. That is,
$$
\alpha_n := \left| B_1 \right| = \frac{\pi^{n/2}}{\Gamma((n + 2)/2)}.
$$
Accordingly, the~$(n - 1)$-dimensional Hausdorff measure of the sphere~$\partial B_1$ is then given by~$\Haus^{n - 1}(\partial B_1) = n \alpha_n$.}
\begin{equation} \label{logI1}
\begin{aligned}
I_1 & = \int_{B_{2 r}(z)} \int_{B_{2 r}(z)} \left[ \left| \eta(x) - \eta(y) \right|^2 - \frac{\left| \eta(y) u(x) - \eta(x) u(y) \right|^2}{u(x) u(y)} \right] K(x, y) \, dx dy \\
& \le \frac{16 \Lambda}{r^2} \int_{B_{2 r}(z)} \int_{B_{2 r}(z)} \frac{dx dy}{|x - y|^{n - 2 + 2 s}} - \lambda \int_{B_r(z)} \int_{B_r(z)} \frac{\left| u(x) - u(y) \right|^2}{u(x) u(y)} \, \frac{dx dy}{|x - y|^{n + 2 s}} \\
& \le 2^{n + 4} n \alpha_n^2 \Lambda r^{n - 2} \int_0^{4 r} \rho^{1 - 2 s} \, d\rho - \lambda \int_{B_r(z)} \int_{B_r(z)} \frac{\left| \log u(x) - \log u(y) \right|^2}{|x - y|^{n + 2 s}} \, dx dy \\
& \le \frac{2^{n + 7} n \alpha_n^2 \Lambda}{s_0} \, r^{n - 2 s} - \lambda [\log u]_{H^s(B_r(z))}^2.
\end{aligned}
\end{equation}
On the other hand, by the non-negativity of~$u$ and again~\eqref{Kbounds} we estimate
\begin{equation} \label{logI2}
\begin{aligned}
I_2 & = \int_{B_{3 r / 2}(z)} \frac{\eta^2(y)}{u(y)} \left( \int_{\R^n \setminus B_{2 r}(z)} \left( u(y) - u(x) \right) K(x, y) \, dx \right) dy \\
& \le \Lambda \int_{B_{3 r / 2}(z)} \eta^2(y) \left( \int_{\R^n \setminus B_{2 r}(z)} |x - y|^{- n - 2 s} \, dx \right) dy \\
& \le \frac{2^{3 n + 1} n \alpha_n^2 \Lambda}{s_0} r^{n - 2 s}.
\end{aligned}
\end{equation}
Finally, using~\eqref{ugef+delta} we have
\begin{align*}
\int_{B_{3 r / 2}(z)} \frac{f(x) \eta^2(x)}{u(x)} \, dx \ge - \int_{B_{3 r / 2}(z)} \frac{|f(x)|}{u(x)} \, dx \ge - \frac{ \| f \|_{L^\infty(B_1)} |B_{3 r / 2}|}{\| f \|_{L^\infty(B_1)} + \delta} \ge - 2^n \alpha_n r^{n - 2 s},
\end{align*}
since~$r < 1$. Claim~\eqref{logclaim} then follows by combining this last equation with~\eqref{8bis},
\eqref{logI1} and~\eqref{logI2}.

We are now ready to show that~$\log u \in BMO(B_{1 / 2})$. For a bounded~$\Omega \subset \R^n$ and~$v \in L^1(\Omega)$, write
$$
(v)_\Omega := \frac{1}{|\Omega|} \int_\Omega v(x) \, dx.
$$
Applying both H\"{o}lder's and fractional Poincar\'e's inequality, from~\eqref{logclaim} we obtain
\begin{align*}
\| \log u - (\log u)_{B_r(z)} \|_{L^1(B_r(z))} & \le |B_r|^{1 / 2} \| \log u - (\log u)_{B_r(z)} \|_{L^2(B_r(z))} \\
& \le c_2 r^{s + n / 2} \, [\log u]_{H^s(B_r(z))} \\
& \le c_3 r^n,
\end{align*}
for some~$c_2, c_3 > 0$ which may depend on~$n$,~$s_0$,~$\lambda$ and~$\Lambda$. Since the above inequality holds for any~$B_r(z) \subseteq B_{1 / 2}$, we conclude that~$\log u \in BMO(B_{1 / 2})$.

Estimate~\eqref{up0est} then follows by the John-Nirenberg embedding in one of its equivalent forms (see, for instance, Theorem~6.25 of~\cite{GM12}). Observe that the exponent~$p_\star$ given by such result is of the form of a dimensional constant divided by the~$BMO(B_{1 / 2})$ semi-norm of~$\log u$. This norm being bounded from above by~$c_3$ and since we are free to make~$p_\star$ smaller if necessary, it turns out that we can choose~$p_\star \in (0, 1)$ to depend only on~$n$,~$s_0$,~$\lambda$ and~$\Lambda$.
\end{proof}

Next is the step of the proof in which the iterative argument really comes into play.

\begin{lemma} \label{moserlem}
Let~$f \in L^\infty(B_1)$ and~$u \in X(B_1)$ be a supersolution of~\eqref{ELeq} in~$B_1$. Assume that~$u$ satisfies~\eqref{ugef+delta}, for some~$\delta > 0$. Then, for any~$p_0 > 0$,
\begin{equation} \label{uinfp0est}
\inf_{B_{1 / 4}} u \ge c_\sharp \left( \dashint_{B_{1 / 2}} u(x)^{- p_0} \, dx \right)^{- 1 / p_0},
\end{equation}
for some constant~$c_\sharp > 0$ which may depend on~$n$,~$s_0$,~$\lambda$,~$\Lambda$ and~$p_0$.
\end{lemma}
\begin{proof}
Fix~$\theta \in (0, 1)$. We claim that, for any~$r \in (0, 1 / 2]$ and~$p > 1$, it holds
\begin{equation} \label{prebrick}
\int_{B_{\theta r}} \int_{B_{\theta r}} \frac{\left| u(x)^{(- p + 1) / 2} - u(y)^{(- p + 1) / 2} \right|^2}{|x - y|^{n + 2 s}} \, dx dy \le c_1 \frac{p^2}{(1 - \theta)^{2} r^{2 s}} \int_{B_r} u(x)^{- p + 1} \, dx,
\end{equation}
for some constant~$c_1 > 0$ depending on~$n$,~$s_0$,~$\lambda$ and~$\Lambda$.

To prove~\eqref{prebrick}, consider a cut-off~$\eta \in C^\infty_c(\R^n)$ such that~$0 \le \eta \le 1$ in~$\R^n$,~$\supp(\eta) = B_r$,~$\eta = 1$ in~$B_{\theta r}$ and~$|\nabla \eta| \le 2 (1 - \theta)^{-1} r^{-1}$ in~$\R^n$, and plug~$\varphi := \eta^{p + 1} u^{- p}$ into~\eqref{ELsupsol}. Inequality~\eqref{prebrick} then follows by arguing as in Lemma~3.5 of~\cite{K09} and noticing that, by~\eqref{ugef+delta},
$$
\int_{B_r} \frac{f(x) \eta(x)^{p + 1}}{u(x)^p} \, dx \ge - \int_{B_r} \frac{|f(x)| u(x)^{- p + 1}}{u(x)} \, dx \ge - r^{- 2 s} \int_{B_r} u(x)^{- p + 1} \, dx,
$$
where we also used the fact that~$r < 1$.

By using~\eqref{prebrick} in combination with the fractional Sobolev inequality, we then deduce
\begin{equation} \label{brick}
\left( \dashint_{B_{\theta r}} u(x)^{\frac{n (- p + 1)}{n - 2 s}} \, dx \right)^{(n - 2 s) / n} \le c_2 \frac{p^2}{(1 - \theta)^{2} \theta^n} \, \dashint_{B_r} u(x)^{- p + 1} \, dx,
\end{equation}
for some~$c_2 \ge 1$ which depends only on~$n$,~$s_0$,~$\lambda$ and~$\Lambda$.

We are now in position to run the iterative scheme, which is based on the fundamental estimate~\eqref{brick}. For any~$k \in \N \cup \{ 0 \}$, define
\begin{align*}
r_k := \frac{1 + 2^{-k}}{4}, \quad p_k := \left( \frac{n}{n - 2 s} \right)^k p_0 \quad \mbox{and} \quad \Phi_k := \left( \dashint_{B_{r_k}} u(x)^{- p_k} \, dx \right)^{1 / p_k},
\end{align*}
so that
$$
\theta_k := \frac{r_{k + 1}}{r_k} = \frac{1 + 2^{- k - 1}}{1 + 2^{- k}} \in \left[ \frac{3}{4}, 1 \right).
$$
We apply~\eqref{brick} with~$r = r_k$,~$\theta = \theta_k$ and~$p = 1 + p_k$, to get
\begin{equation} \label{iter}
\Phi_{k + 1} \le q_k \Phi_k,
\end{equation}
for any~$k \in \N \cup \{ 0 \}$, where
$$
q_k := \left[ c_2 \frac{(1 + p_k)^2}{(1 - \theta_k)^{2} \theta_k^n} \right]^{1 / p_k}.
$$
From~\eqref{iter} it then follows that
\begin{equation}\label{Phi k pro}
\Phi_k \le \Phi_0 \prod_{j = 0}^{k - 1} q_j.
\end{equation}
Now we observe that
$$
1 - \theta_k = \frac{2^{-k}-2^{-k-1}}{1+2^{-k}} = \frac{1}{2^{k+1}+2}
\ge \frac{1}{2^{k+2}}.
$$
Therefore, recalling that~$\theta_k \ge 3 / 4$,
$$
\frac{1}{(1 - \theta_k)^2 \theta_k^n} \le 2^{2 ( k + 2 )} \left( \frac{4}{3} \right)^n \le 2^{2 k + n + 4},
$$
and hence
$$
\log q_k \le \frac{1}{p_k} \log \left[ c_2 (1 + p_k)^2 2^{2 k + n + 4} \right] \le \frac{1}{p_k} \log \left[ c_3 \left( \frac{2 n}{n - 2 s} \right)^{2 k} \right] \le c_4 \left( \frac{n - 2 s_0}{n} \right)^k k,
$$
for some~$c_3, c_4 > 0$ that may also depend on~$p_0$. This implies that the product of the~$q_j$'s converges, as~$k \rightarrow +\infty$. Thence,~\eqref{uinfp0est} follows
from~\eqref{Phi k pro}, since
\begin{equation*}
\liminf_{k \rightarrow +\infty} \Phi_k \ge \lim_{k \rightarrow +\infty} |B_{r_k}|^{- 1 / p_k} \| u^{-1} \|_{L^{p_k}(B_{1 / 4})} = \sup_{B_{1 / 4}} u^{-1} = \left( \inf_{B_{1 / 4}} u \right)^{-1}. \qedhere
\end{equation*}
\end{proof}

By putting together Lemmata~\ref{JNlem} and~\ref{moserlem}, we easily obtain the following \emph{weak Harnack inequality}.

\begin{corollary} \label{weakharnackcor}
Let~$r \in (0, 1]$ and~$f \in L^\infty(B_r)$. Assume that~$u \in X(B_r) \cap L^\infty(\R^n)$ is a non-negative supersolution of~\eqref{ELeq} in~$B_r$. Then,
\begin{equation} \label{weakharnack}
\inf_{B_{r / 4}} u + r^{2 s} \| f \|_{L^\infty(B_r)} \ge c_\star \left( \dashint_{B_{r / 2}} u(x)^{p_\star} \right)^{1 / p_\star},
\end{equation}
for some~$c_\star \in (0, 1)$ depending only on~$n$,~$s_0$,~$\lambda$ and~$\Lambda$.
\end{corollary}
\begin{proof}
Assume for the moment~$r = 1$. Let then~$\delta > 0$ be a small parameter and define~$u_\delta := u + \| f \|_{L^\infty(B_1)} + \delta$. Note that~$u_\delta$ is still a non-negative supersolution of~\eqref{ELeq} in~$B_1$ and that it satisfies~\eqref{ugef+delta}. Thus, we are free to apply Lemmata~\ref{JNlem} and~\ref{moserlem} to~$u_\delta$ and 
obtain that
$$
\inf_{B_{1 / 4}} u + \| f \|_{L^\infty(B_1)} + \delta 
\ge \frac{c_\sharp}{C_\star} \left( \dashint_{B_{1 / 2}} u(x)^{p_\star} \, dx \right)^{1 / p_\star}.
$$
Letting~$\delta \rightarrow 0^+$ we obtain~\eqref{weakharnack} when~$r = 1$. For a general radius~$r \le 1$ the result follows by a simple scaling argument. 
\end{proof}

With the aid of Corollary~\ref{weakharnackcor}, we can prove the following proposition, which will be the fundamental step in the conclusive inductive argument. In the literature, results of this kind are often known as~\emph{growth lemmata}.

\begin{proposition} \label{keyprop}
There exist~$\gamma \in (0, 2 s_0)$ and~$\eta \in (0, 1)$, depending only on~$n$,~$s_0$,~$\lambda$ and~$\Lambda$, such that for any~$r \in (0, 1]$,~$f \in L^\infty(B_r)$ and~$u \in X(B_r) \cap L^\infty(\R^n)$ supersolution of~\eqref{ELsupsol} in~$B_r$, for which
\begin{equation} \label{uge0key}
u(x) \ge 0 \quad \mbox{for a.e. } x \in B_{2 r},
\end{equation}
\begin{equation} \label{ugehalfkey}
\left| \left\{ x \in B_{r / 2} : u(x) \ge 1 \right\} \right| \ge \frac{1}{2} |B_{r / 2}|,
\end{equation}
and
\begin{equation} \label{uoutside}
u(x) \ge - 2 \left( 8 \frac{|x|}{2 r} \right)^\gamma + 2 \quad \mbox{for a.e. } x \in \R^n \setminus B_{2 r},
\end{equation}
hold true, then
\begin{equation} \label{keyine}
\inf_{B_{r / 4}} u + r^{2 s} \| f \|_{L^\infty(B_r)} \ge \eta.
\end{equation}
\end{proposition}
\begin{proof}
Write~$u = u_+ - u_-$. Using~\eqref{Ksimm} and~\eqref{uge0key}, it is easy to see that~$u_+$ is a supersolution of
$$
\D_K(u_+, \cdot) = \tilde{f} \quad \mbox{in } B_{r},
$$
where
$$
\tilde{f}(x) := f(x) - 2 \int_{\R^n \setminus B_{2 r}} u_-(y) K(x, y) \, dy.
$$
Applying Corollary~\ref{weakharnackcor} we get that
$$
\inf_{B_{r / 4}} u_+ + r^{2 s} \| \tilde{f} \|_{L^\infty(B_r)} \ge c_\star \left( \dashint_{B_{r / 2}} u_+(x)^{p_\star} \right)^{1 / p_\star}.
$$
Using then hypotheses~\eqref{uge0key} and~\eqref{ugehalfkey}, this yields
\begin{equation} \label{keytech}
\begin{aligned}
\inf_{B_{r / 4}} u + r^{2 s} \| \tilde{f} \|_{L^\infty(B_r)} & \ge c_\star \left( \dashint_{B_{r / 2} \cap \{ u \ge 1 \}} u(x)^{p_\star} \right)^{1 / p_\star} \\
& \ge c_\star \left( \frac{\left| \left\{ x \in B_{r / 2} : u(x) \ge 1  \right\} \right|}{|B_{r / 2}|} \right)^{1 / p_\star} \\
& \ge c_\star 2^{- 1 / p_\star} =: 2 \eta.
\end{aligned}
\end{equation}

Now we turn our attention to the~$L^\infty$ norm of~$\tilde{f}$. First, we notice that~\eqref{uoutside} implies that
$$
u_-(x) \le 2 \left( 8 \frac{|x|}{2 r} \right)^\gamma - 2 \quad \mbox{for a.e. } x \in \R^n \setminus B_{2 r},
$$
as the right-hand side of~\eqref{uoutside} is negative. Moreover, given~$x \in B_r$ and~$y \in \R^n \setminus B_{2 r}$, it holds
$$
|y - x| \ge |y| - |x| \ge |y| - \frac{|y|}{2} = \frac{|y|}{2}.
$$
Consequently, recalling~\eqref{Kbounds} we compute
\begin{align*}
\int_{\R^n \setminus B_{2 r}} u_-(y) K(x, y) \, dy & \le \Lambda \int_{\R^n \setminus B_{2 r}} \frac{2 \left( 8 \frac{|y|}{2 r} \right)^\gamma - 2}{|x - y|^{n + 2 s}} \, dy \\
& \le 2^{n + 2 s + 1} \Lambda \left[ \left( \frac{4}{r} \right)^\gamma \int_{\R^n \setminus B_{2 r}} |y|^{\gamma - n - 2 s} \, dy - \int_{\R^n \setminus B_{2 r}} |y|^{- n - 2 s} \, dy \right] \\
& = 2^{n + 1} n \alpha_n \Lambda \left[ \frac{8^\gamma}{2 s - \gamma} - \frac{1}{2 s} \right] r^{- 2 s},
\end{align*}
if~$\gamma < 2 s_0$. Notice that the term in brackets on the last line of the above formula converges to~$0$ as~$\gamma \rightarrow 0^+$, uniformly in~$s \ge s_0$. Therefore, we can find~$\gamma > 0$, in dependence of~$n$,~$s_0$,~$\lambda$ and~$\Lambda$, such that
$$
\| \tilde{f} \|_{L^\infty(B_r)} \le \| f \|_{L^\infty(B_r)} + r^{- 2 s} \eta.
$$
Inequality~\eqref{keyine} then follows by combining this with~\eqref{keytech}.
\end{proof}

We are now ready to move to the actual

\begin{proof}[Proof of Theorem~\ref{holderthm}]
We focus on the proof of~\eqref{holderine}, as the H\"{o}lder continuity of~$u$ inside~$\Omega$ would then easily follow. Furthermore, we may assume without loss of generality~$x_0$ to be the origin.

Set
\begin{equation} \label{Rbound}
R_0 := \left( \frac{\eta}{4} \right)^{\frac{1}{2 s_0}} < 1,
\end{equation}
with~$\eta$ as in Proposition~\ref{keyprop}, and take~$R \in (0, R_0]$. We claim that there exist a constant~$\alpha \in (0, 1)$, depending only on~$n$,~$s$,~$\lambda$ and~$\Lambda$, a non-decreasing sequence~$\{ m_j \}$ and a non-increasing sequence~$\{ M_j \}$ of real numbers such that for any~$j \in \N \cup \{ 0 \}$
\begin{equation} \label{ind}
\begin{aligned}
& m_j \le u(x) \le M_j \quad \mbox{for a.e. } x \in B_{8^{- j} R}, \\
& M_j - m_j = 8^{- j \alpha} L,
\end{aligned}
\end{equation}
with
\begin{equation} \label{Ldef}
L := 2 \| u \|_{L^\infty(\R^n)} + \| f \|_{L^\infty(B_R)}.
\end{equation}

We prove this by induction. Set~$m_0 := - \| u \|_{L^\infty(\R^n)}$ and~$M_0 := \| u \|_{L^\infty(\R^n)} + \| f \|_{L^\infty(B_R)}$. With this choice, property~\eqref{ind} clearly holds true for~$j = 0$. Then, for a fixed~$k \in \N$, we assume to have constructed the two sequences~$\{ m_j \}$ and~$\{ M_j \}$ up to~$j = k - 1$ in such a way that~\eqref{ind} is satisfied and show that we can also build~$m_k$ and~$M_k$. For any~$x \in \R^n$, define
$$
v(x) := \frac{2 \cdot 8^{(k - 1) \alpha}}{L} \left( u(x) - \frac{M_{k - 1} + m_{k - 1}}{2} \right),
$$
with
\begin{equation} \label{alphadef}
\alpha := \min \left\{ \gamma, \frac{\log \left( \frac{4}{4 - \eta} \right)}{\log 8} \right\},
\end{equation}
and~$\gamma, \eta$ as in Proposition~\ref{keyprop}. Since~$u$ is a solution of~\eqref{ELeq} in~$\Omega$, we deduce that~$v$ satisfies
\begin{equation} \label{veq}
\D_K(v, \cdot) = \frac{2 \cdot 8^{(k - 1) \alpha}}{L} f \quad \mbox{in } B_{8^{- (k - 1)} R}.
\end{equation}
Moreover,
\begin{equation} \label{|v|le1}
|v(x)| \le 1 \quad \mbox{for a.e. } x \in B_{8^{- (k - 1)} R}.
\end{equation}
Letting instead~$x \in \R^n \setminus B_{8^{- (k - 1)} R}$, there exists a unique~$\ell \in \N$ for which
$$
8^{- (k - \ell)} R \le |x| < 8^{- (k - \ell - 1)} R.
$$
Writing~$m_{- j} := m_0$ and~$M_{- j} := M_0$ for every~$j \in \N$, we compute
\begin{equation} \label{vle}
\begin{aligned}
v(x) & \le \frac{2 \cdot 8^{(k - 1) \alpha}}{L} \left( M_{k - \ell - 1} - m_{k - \ell - 1} + m_{k - \ell - 1} - \frac{M_{k - 1} + m_{k - 1}}{2} \right) \\
& \le \frac{2 \cdot 8^{(k - 1) \alpha}}{L} \left( M_{k - \ell - 1} - m_{k - \ell - 1} - \frac{M_{k - 1} - m_{k - 1}}{2} \right) \\
& \le \frac{2 \cdot 8^{(k - 1) \alpha}}{L} \left( 8^{- (k - \ell - 1) \alpha} L - \frac{8^{- (k - 1) \alpha} L}{2} \right) \\
& = 2 \cdot 8^{\ell \alpha} - 1 \\
& \le 2 \left( 8 \frac{|x|}{8^{- (k - 1)} R} \right)^\alpha - 1,
\end{aligned}
\end{equation}
Analogously, one checks that
\begin{equation} \label{vge}
v(x) \ge - 2 \left( 8 \frac{|x|}{8^{-(k - 1)} R} \right)^\alpha + 1,
\end{equation}
for a.e.~$x \in \R^n \setminus B_{8^{- (k - 1)} R}$.

We distinguish between the two mutually exclusive possibilities
\begin{enumerate}[(a)]
\item $\left| \left\{ x \in B_{8^{- (k - 1)} R / 4} : v(x) \le 0 \right\} \right| \ge \frac{1}{2} |B_{8^{- (k - 1)} R / 4}|$, and
\item $\left| \left\{ x \in B_{8^{- (k - 1)} R / 4} : v(x) \le 0 \right\} \right| < \frac{1}{2} |B_{8^{- (k - 1)} R / 4}|$.
\end{enumerate}
In case~(a), set~$\tilde{u} := 1 - v$. From~\eqref{veq} we deduce in particular that
$$
\D_K(\tilde{u}, \cdot) = - \frac{2 \cdot 8^{(k - 1) \alpha}}{L} f \quad \mbox{in } B_{8^{-(k - 1)} R / 2}.
$$
In view of~\eqref{|v|le1} and~\eqref{vle} we apply Proposition~\ref{keyprop} to~$\tilde{u}$, with~$r = 8^{- (k - 1)} R / 2$, and obtain that
$$
\inf_{B_{8^{- (k - 1)} R / 8}} \tilde{u} + \left( \frac{8^{- (k - 1)} R}{2} \right)^{2 s} \left\| - \frac{2 \cdot 8^{(k - 1) \alpha}}{L} f \right\|_{L^\infty(B_{8^{- (k - 1)} R / 2})} \ge \eta,
$$
from which, by~\eqref{Ldef} and~\eqref{Rbound}, it follows
\begin{align*}
\sup_{B_{8^{- k} R}} v & \le 1 - \eta + \left( \frac{8^{- (k - 1)} R}{2} \right)^{2 s} \left\| - \frac{2 \cdot 8^{(k - 1) \alpha}}{L} f \right\|_{L^\infty(B_{8^{- (k - 1)} R / 2})} \\
& \le 1 - \eta + 2 \cdot 8^{- (2 s_0 - \alpha) (k - 1)} R_0^{2 s_0} \frac{\| f \|_{L^\infty(B_R)}}{L} \\
& \le 1 - \frac{\eta}{2}.
\end{align*}
Note that we took advantage of the fact that~$\alpha \le \gamma < 2 s_0$, by~\eqref{alphadef}. If we translate this estimate back to~$u$, applying~\eqref{alphadef} once again we finally get
\begin{align*}
\sup_{B_{8^{- k} R}} u & \le \left( 1 - \frac{\eta}{2} \right) \frac{L}{2 \cdot 8^{(k - 1) \alpha}} + \frac{M_{k - 1} + m_{k - 1}}{2} \\
& = \left( 1 - \frac{\eta}{2} \right) \frac{M_{k - 1} - m_{k - 1}}{2} + \frac{M_{k - 1} + m_{k - 1}}{2} \\
& = m_{k - 1} + \left( \frac{4 - \eta}{4} \right) \left( M_{k - 1} - m_{k - 1} \right) \\
& \le m_{k - 1} + 8^{- k \alpha} L.
\end{align*}
Accordingly,~\eqref{ind} is satisfied by setting~$m_k := m_{k - 1}$ and~$M_k := m_{k - 1} + 8^{- k \alpha} L$.

If on the other hand~(b) holds we define~$\tilde{u} := 1 + v$. With a completely analogous argument using~\eqref{vge} in place of~\eqref{vle}, we end up estimating
$$
\inf_{B_{8^{-k} R}} u \ge M_{k - 1} - 8^{- k \alpha} L,
$$
so that~\eqref{ind} again follows with~$m_k := M_{k - 1} - 8^{- k \alpha} L$ and~$M_k := M_{k - 1}$.

The proof of the theorem is therefore complete, as the bound in~\eqref{holderine} is an immediate consequence of claim~\eqref{ind}.
\end{proof}

\section{An energy estimate} \label{enestsec}

We include here a result which addresses the growth of the energy~$\E$ of local minimizers inside large balls. We point out that, as in Section~\ref{regsec}, this estimate is set in a general framework. In particular, the periodicity of~$K$ and~$W$ encoded in~\eqref{Kper} and~\eqref{Wper} is not significant here. Writing
\begin{equation} \label{Psidef}
\Psi_s(R) := \begin{cases}
R^{1 - 2 s} & \quad \mbox{if } s \in (0, 1/2) \\
     \log R & \quad \mbox{if } s = 1/2 \\
          1 & \quad \mbox{if } s \in (1/2, 1),
\end{cases}
\end{equation}
we can state the following

\begin{proposition} \label{enestprop}
Let~$n \in \N$,~$s \in (0, 1)$,~$x_0 \in \R^n$ and~$R \ge 3$. Assume that~$K$ and~$W$ satisfy\footnote{We observe that,
at this level, only the boundedness of~$W$ encoded
in~\eqref{Wbound} is relevant here. Thus, no assumption on the derivative~$W_r$ is necessary. See in particular the proof of
Proposition~\ref{enestprop}.}
\eqref{Ksimm},~\eqref{Kbounds} and~\eqref{Wzeros},~\eqref{Wbound}, respectively. If~$u: \R^n \to [-1, 1]$ is a local minimizer of~$\E$ in~$B_{R + 2}(x_0)$, then
\begin{equation} \label{enest}
\E(u; B_R(x_0)) \le C R^{n - 1} \Psi_s(R),
\end{equation}
for some constant~$C > 0$ which depends on~$n$,~$s$,~$\Lambda$ and~$W^*$.
\end{proposition}

The above proposition will play an important
role later in Subsection~\ref{unconstrsubsec}, as it will imply that the interface region of a minimizer cannot be too wide.


Estimate~\eqref{enest} has first been proved in~\cite{CC14} and~\cite{SV14} for the fractional Laplacian. While in the first paper the authors use the harmonic extension of~$u$ to~$\R^{n + 1}_+$ to prove~\eqref{enest}, in the latter work the result is obtained by explicitly computing the energy~$\E$ of a suitable competitor of~$u$. It turns out that this strategy is flexible enough to be adapted to our framework and the proof of Proposition~\ref{enestprop} is actually an appropriate and careful
modification of that of~\cite[Theorem~1.3]{SV14}.

\medskip

Before heading to the proof of Proposition~\ref{enestprop}, we first need the following auxiliary result that will be also widely used in the following Section~\ref{s>1/2sec}.

\begin{lemma} \label{uvmMlem}
Let~$U, V$ be two measurable subsets of~$\R^n$ and~$u, v \in H^s_\loc(\R^n)$. Then,
\begin{equation} \label{KuvmM}
\K(\min \{ u, v \}; U, V) + \K( \max \{ u, v \}; U, V) \le \K(u; U, V) + \K(v; U, V),
\end{equation}
and
\begin{equation} \label{PuvmM}
\P(\min \{ u, v \}; U) + \P(\max \{ u, v\}; U) = \P(u; U) + \P(v; V).
\end{equation}
\end{lemma}
\begin{proof}
Since the derivation of identity~\eqref{PuvmM} is quite straightforward, we focus on~\eqref{KuvmM} only.

We write for simplicity~$m := \min \{ u, v \}$ and~$M := \max \{ u, v \}$. Observe that we may assume the right-hand side of~\eqref{KuvmM} to be finite, the result being otherwise obvious.
In order to show~\eqref{KuvmM}, we actually prove the stronger pointwise relation
\begin{equation} \label{uvmMpw}
|m(x) - m(y)|^2 + |M(x) - M(y)|^2 \le |u(x) - u(y)|^2 + |v(x) - v(y)|^2,
\end{equation}
for a.e.~$x, y \in \R^n$.

Let then~$x$ and~$y$ be two fixed points in~$\R^n$. In order to check that~\eqref{uvmMpw} is true, we consider separately the two possibilities
\begin{enumerate}[i)]
\item $u(x) \le v(x)$ and~$u(y) \le v(y)$, or~$u(x) > v(x)$ and~$u(y) > v(y)$;
\item $u(x) \le v(x)$ and~$u(y) > v(y)$, or~$u(x) > v(x)$ and~$u(y) \le v(y)$.
\end{enumerate}
In the first situation it is immediate to see that~\eqref{uvmMpw} holds as an identity. Suppose then that point~ii) occurs. If this is the case, we compute
\begin{align*}
& |m(x) - m(y)|^2 + |M(x) - M(y)|^2 \\
& \hspace{60pt} = |u(x) - v(y)|^2 + |v(x) - u(y)|^2 \\
& \hspace{60pt} = |u(x) - u(y)|^2 + |v(x) - v(y)|^2 + 2 \left( u(x) - v(x) \right) \left( u(y) - v(y) \right) \\
& \hspace{60pt} \le |u(x) - u(y)|^2 + |v(x) - v(y)|^2,
\end{align*}
which is~\eqref{uvmMpw}. The proof of the lemma is thus complete.
\end{proof}

\begin{proof}[Proof of Proposition~\ref{enestprop}]
Without loss of generality, we assume~$x_0$ to be the origin. In the course of the proof we will denote as~$c$ any positive constant which depends at most on~$n$,~$s$,~$\Lambda$ and~$W^*$.

Let~$\psi$ be the radially symmetric function defined by
$$
\psi(x) := 2 \min \left\{ (|x| - R - 1)_+, 1 \right\} - 1 = \begin{cases}
-1              & \mbox{ if } x \in B_{R + 1} \\
2 |x| - 2 R - 1 & \mbox{ if } x \in B_{R + 2} \setminus B_{R + 1} \\
1               & \mbox{ if } x \in \R^n \setminus B_{R + 2}.
\end{cases}
$$
We claim that~$\psi$ satisfies~\eqref{enest} in~$B_{R + 2}$, that is
\begin{equation} \label{psienest}
\E(\psi; B_{R + 2}) \le c R^{n - 1} \Psi_s(R).
\end{equation}
Indeed, let~$x \in B_{R + 2}$ and set~$d(x) := \max \{ R - |x|, 1 \}$. It is easy to see that
$$
\left| \psi(x) - \psi(y) \right| \le 2 \begin{cases}
d(x)^{-1}|x - y| & \quad \mbox{if } |x - y| < d(x) \\
1                & \quad \mbox{if } |x - y| \ge d(x).
\end{cases}
$$
Consequently, applying~\eqref{Kbounds} we compute
\begin{align*}
\int_{\R^n} |\psi(x) - \psi(y)|^2 K(x, y) \, dy & \le 4 \omega_{n - 1} \Lambda \left[ d(x)^{-2} \int_0^{d(x)} \rho^{1 - 2 s} \, d\rho + \int_{d(x)}^{+\infty} \rho^{- 1 - 2 s} \, d\rho \right] \\
& \le c d(x)^{- 2 s}.
\end{align*}
Furthermore, using polar coordinates we get
\begin{equation} \label{d(x)est}
\begin{aligned}
\int_{B_{R + 2}} d(x)^{- 2 s} \, dx & = \int_{B_{R - 1}} \frac{dx}{\left( R - |x| \right)^{2 s}} + \int_{B_{R + 2} \setminus B_{R - 1}} dx \le c R^{n - 1} \Psi_s(R).
\end{aligned}
\end{equation}
Hence,
$$
\int_{B_{R + 2}} \int_{\R^n} |\psi(x) - \psi(y)|^2 K(x, y) \, dx dy \le c R^{n - 1} \Psi_s(R).
$$
Since by~\eqref{Wbound} and~\eqref{Wzeros} we also have
$$
\P(\psi, B_{R + 2}) = \int_{B_{R + 2}} W(x, \psi(x)) \, dx \le W^* \int_{B_{R + 2} \setminus B_{R + 1}} dx \le c R^{n - 1},
$$
it is clear that estimate~\eqref{psienest} follows.

Now, set~$v := \min \{ u, \psi \}$ and~$w := \max \{ u, \psi \}$. By the definition of~$\psi$ and the fact that~$-1 \le u \le 1$, we observe that
\begin{equation} \label{u=venest}
u = v \quad \mbox{in } \R^n \setminus B_{R + 2},
\end{equation}
and
\begin{equation} \label{u=wenest}
u = w \quad \mbox{in } B_{R + 1}.
\end{equation}
By virtue of~\eqref{u=wenest},
\begin{equation} \label{KP=}
\K(u; B_R, B_R) = \K(w; B_R, B_R) \quad \mbox{and} \quad \P(u; B_R) = \P(w; B_R).
\end{equation}
On the other hand, we claim that
\begin{equation} \label{u-w}
\K(u; B_R, \R^n \setminus B_R) \le \K(w; B_R, \R^n \setminus B_R) + c R^{n - 1} \Psi_s(R).
\end{equation}
Indeed, using~\eqref{Kbounds},~\eqref{u=wenest} and the fact that~$|u|, |\psi| \le 1$ a.e. in~$\R^n$, we compute
\begin{align*}
& \K(u; B_R, \R^n \setminus B_R) - \K(w; B_R, \R^n \setminus B_R) \\
& \hspace{50pt} = \frac{1}{2} \int_{B_R} \left( \int_{\R^n \setminus B_{R + 1}} \left[ |u(x) - u(y)|^2 - |u(x) - w(y)|^2 \right] K(x, y) \, dy \right) dx \\
& \hspace{50pt} \le 2 \Lambda \int_{B_R} \left( \int_{\R^n \setminus B_{R + 1}} |x - y|^{- n - 2 s} \, dy \right) dx  \le c \int_{B_R} d(x)^{- 2 s} \, dx,
\end{align*}
and claim~\eqref{u-w} then follows from~\eqref{d(x)est}. Accordingly, by~\eqref{u-w} and~\eqref{KP=} we obtain that
\begin{equation} \label{Euwenest}
\E(u; B_R) \le \E(w; B_R) + c R^{n - 1} \Psi_s(R).
\end{equation}

We now take advantage of the minimality of~$u$ and~\eqref{u=venest} to deduce
$$
\E(u; B_{R + 2}) \le \E(v; B_{R + 2}).
$$
Then, from this and Lemma~\ref{uvmMlem} it follows immediately that
\begin{equation} \label{Ewpsienest}
\E(w; B_R) \le \E(w; B_{R + 2}) \le \E(\psi; B_{R + 2}).
\end{equation}
Note that the first inequality above is true as a consequence of the inclusion~$\C_{B_R} \subset \C_{B_{R + 2}}$ (see Remark~\ref{restrmk}). By applying in sequence~\eqref{Euwenest},~\eqref{Ewpsienest} and~\eqref{psienest}, we finally get~\eqref{enest}.
\end{proof}

\section{Proof of Theorem~\ref{mainthm} for rapidly decaying kernels} \label{s>1/2sec}

\counterwithin{definition}{subsection}
\counterwithin{equation}{subsection}

The present section contains the proof of Theorem~\ref{mainthm} under the additional assumption that~$K$ satisfies
\begin{equation} \tag{K4} \label{Krapid}
K(x, y) \le \frac{\Gamma}{|x - y|^{n + \beta}} \quad \mbox{for a.e. } x, y \in \R^n \mbox{ such that } |x - y| \ge \bar{R}, \mbox{ with } \beta > 1,
\end{equation}
for some constants~$\Gamma, \bar{R} > 0$. We stress that this hypothesis is merely technical and in fact it will be removed later in Section~\ref{s<1/2sec}. However, we need the fast decay of the kernel~$K$ at infinity - ensured by the fact that~$\beta > 1$ - in order to perform a delicate construction at some point
(roughly speaking, the decay assumed in~\eqref{Krapid}
is needed to ensure the existence of a competitor with
finite energy in the large, but the geometric estimates
will be independent of the quantities
in~\eqref{Krapid} and this will allow
us to perform a limit procedure). Hence, we assume~\eqref{Krapid} to hold in the whole section.

Notice that if~$s > 1/2$, then~\eqref{Krapid} is automatically fulfilled in view of~\eqref{Kbounds}.

\medskip

The argument leading to the proof
of Theorem~\ref{mainthm} is long and articulated. Therefore, we divide the section into several subsections which we hope will make the reading easier.

We first deal with the case of a rational direction~$\omega$. Under this assumption, we can take advantage of the equivalence relation~$\sim_\omega$ defined in~\eqref{simrel} to build the minimizer. This construction occupies Subsections~\ref{minpersubsec}-\ref{unconstrsubsec}.

Irrational directions - i.e.~$\omega \in \R^n \setminus \Q^n$ - are then treated in Subsection~\ref{irrsubsec} as limiting cases.

\medskip

For simplicity of exposition, we restrict ourselves to consider~$\theta = 9 / 10$. The general case is in no way different. Of course, the choice~$9 / 10$ is made in order to represent a value of~$\theta$ close to~$1$.

\subsection{Minimization with respect to periodic perturbations} \label{minpersubsec}

Let~$\omega \in \Q^n \setminus \{ 0 \}$ be fixed. Given a measurable function~$u: \R^n \to \R$, we say that~$u \in L^2_\loc(\tRn)$ if~$u \in L^2_\loc(\R^n)$ and~$u$ is periodic with respect to~$\sim$. Given~$A < B$, let
$$
\AAB := \left\{ u \in L^2_\loc(\tRn) : u(x) \ge \frac{9}{10} \mbox{ if } \omega \cdot x \le A \mbox{ and } u(x) \le - \frac{9}{10} \mbox{ if } \omega \cdot x \ge B \right\},
$$
be the set of admissible functions. We introduce the auxiliary functional
\begin{equation} \label{perfunc}
\begin{aligned}
\F(u) := & \, \K(u; \tRn, \R^n) + \P(u; \tRn) \\
= & \, \frac{1}{2} \int_{\tRn} \int_{\R^n} |u(x) - u(y)|^2 K(x, y) \, dx dy + \int_{\tRn} W(x, u(x)) \, dx.
\end{aligned}
\end{equation}
Note that in the integrals above,~$\tRn$ stands for any fundamental domain of the relation~$\sim$. In the following, we will often identify quotients with any of their respective fundamental domains.

\medskip

The aim of this subsection is to prove the existence of an \emph{absolute minimizer} of~$\F$ within the class~$\AAB$, that is a function~$u \in \AAB$ such that~$\F(u) \le \F(v)$, for any~$v \in \AAB$. Such minimizers are the building blocks of our construction, as will become clear in the sequel.

As a first step toward this goal, we show that~$\F$ is not identically infinite on~$\AAB$.

\begin{lemma} \label{finfunlem}
Let~$\bar{u} \in \AAB$ be defined by setting~$\bar{u}(x) := \bar{\mu}(\omega \cdot x)$, where~$\bar{\mu}$ is the piecewise linear function given by
$$
\bar{\mu}(t) := \begin{cases}
1                                        & \mbox{if } t \le A \\
1 - \frac{2}{B - A} \left( t - A \right) & \mbox{if } A < t \le B \\
- 1                                      & \mbox{if } t > B.
\end{cases}
$$
Then,~$\F(\bar{u}) < +\infty$.
\end{lemma}
\begin{proof}
Since~$W(x, \cdot)$ vanishes at~$\pm 1$, for a.e.~$x \in \R^n$, it is clear that the potential term of~$\F$ evaluated at~$\bar{u}$ is finite. Thus, we only need to estimate the kinetic term. To do this, by~\eqref{Kbounds} and~\eqref{Krapid}, it is in turn sufficient to show that
\begin{equation} \label{finkinclaim}
\int_{\tRn} \left( \int_{B_{\bar{R}}(x)} \frac{|\bar{u}(x) - \bar{u}(y)|^2}{|x - y|^{n + 2 s}} \, dy + \int_{\R^n \setminus B_{\bar{R}}(x)} \frac{|\bar{u}(x) - \bar{u}(y)|^2}{|x - y|^{n + \beta}} \, dy \right) dx < +\infty.
\end{equation}
Notice that, up to an affine transformation, we may take~$\omega = e_n$. Moreover, we assume for simplicity that~$A = 0$ and~$B = 1$. In this setting, we have~$\tRn = [0, 1]^{n - 1} \times \R$ and, consequently,~\eqref{finkinclaim} is equivalent to
\begin{equation} \label{finkinclaim2}
I := \int_{[0, 1]^{n - 1} \times \R} \left( \int_{B_{\bar{R}}(x)} \frac{|\bar{u}(x) - \bar{u}(y)|^2}{|x - y|^{n + 2 s}} \, dy \right) dx < +\infty,
\end{equation}
and
\begin{equation} \label{finkinclaim3}
J := \int_{[0, 1]^{n - 1} \times \R} \left( \int_{\R^n \setminus B_{\bar{R}}(x)} \frac{|\bar{u}(x) - \bar{u}(y)|^2}{|x - y|^{n + \beta}} \, dy \right) dx < +\infty.
\end{equation}

By the definition of~$\bar{u}$, it is clear that
$$
I = \int_{[0, 1]^{n - 1} \times [- \bar{R}, \bar{R} + 1]} \left( \int_{B_{\bar{R}}(x)} \frac{|\bar{u}(x) - \bar{u}(y)|^2}{|x - y|^{n + 2 s}} \, dy \right) dx.
$$
Then, we take advantage of~$\bar{u}$ being Lipschitz to compute, using polar coordinates,
$$
I \le 4 \int_{[0, 1]^{n - 1} \times [- \bar{R}, \bar{R} + 1]} \left( \int_{B_{\bar{R}}(x)} \frac{dy}{|x - y|^{n + 2 s - 2}} \right) dx = \frac{2 n \alpha_n}{1 - s} (2 \bar{R} + 1) \bar{R}^{2 - 2 s},
$$
which implies~\eqref{finkinclaim2}.

On the other hand, to prove~\eqref{finkinclaim3} we first write~$J = J_1 + J_2 + J_3$, where
\begin{align*}
J_1 & := \int_{[0, 1]^{n - 1} \times [2, +\infty)} \left( \int_{\R^n \setminus B_{\bar{R}}(x)} \frac{|\bar{u}(x) - \bar{u}(y)|^2}{|x - y|^{n + \beta}} \, dy \right) dx, \\
J_2 & := \int_{[0, 1]^{n - 1} \times (- \infty, - 1]} \left( \int_{\R^n \setminus B_{\bar{R}}(x)} \frac{|\bar{u}(x) - \bar{u}(y)|^2}{|x - y|^{n + \beta}} \, dy \right) dx, \\
J_3 & := \int_{[0, 1]^{n - 1} \times [- 1, 2]} \left( \int_{\R^n \setminus B_{\bar{R}}(x)} \frac{|\bar{u}(x) - \bar{u}(y)|^2}{|x - y|^{n + \beta}} \, dy \right) dx.
\end{align*}
Using the definition of~$\bar{u}$, we observe that
\begin{align*}
J_1 & \le \int_{[0, 1]^{n - 1} \times [2, +\infty)} \left( \int_{\R^{n - 1} \times (- \infty, 1]} \frac{|- 1 - \bar{\mu}(y_n)|^2}{|x - y|^{n + \beta}} \, dy \right) dx \\
& \le 4 \int_{[0, 1]^{n - 1} \times [2, +\infty)} \left( \int_{\R^{n - 1} \times (- \infty, 1]} \frac{dy}{|x - y|^{n + \beta}} \right) dx.
\end{align*}
Making the substitution~$z' := (y' - x') / |x_n - y_n|$, we have
\begin{align*}
\int_{\R^{n - 1} \times (- \infty, 1]} \frac{dy}{|x - y|^{n + \beta}} & = \int_{- \infty}^1 |x_n - y_n|^{- n - \beta} \left[ \int_{\R^{n - 1}} \left( 1 + \frac{|x' - y'|^2}{|x_n - y_n|^2} \right)^{- \frac{n + \beta}{2}} dy' \right] dy_n \\
& = \int_{- \infty}^1 |x_n - y_n|^{- 1 - \beta} \left[ \int_{\R^{n - 1}} \left( 1 + |z'|^2 \right)^{- \frac{n + \beta}{2}} dz' \right] dy_n \\
& = \frac{\Xi}{\beta} (x_n - 1)^{- \beta},
\end{align*}
where we denoted with~$\Xi$ the finite quantity
$$
\int_{\R^{n - 1}} \left( 1 + |z'|^2 \right)^{- \frac{n + \beta}{2}} dz'.
$$
Accordingly,
\begin{align*}
J_1 \le \frac{4 \Xi}{\beta} \int_{2}^{+ \infty} (x_n - 1)^{- \beta} dx_n = \frac{4 \Xi}{(\beta - 1) \beta},
\end{align*}
since~$\beta > 1$. Similarly, one checks that~$J_2$ is finite too. The computation of~$J_3$ is simpler. By taking advantage of the fact that~$\bar{u}$ is a bounded function and switching to polar coordinates, we get
$$
J_3 \le 4 \int_{[0, 1]^{n - 1} \times [- 1, 2]} \left( \int_{\R^n \setminus B_{\bar{R}}(x)} \frac{dy}{|x - y|^{n + \beta}} \right) dx = \frac{12 n \alpha_n}{\beta} \bar{R}^{- \beta}.
$$
Hence,~\eqref{finkinclaim3} follows.
\end{proof}

We want to highlight how crucial condition~\eqref{Krapid} has been in the proof of the above lemma. Indeed, if the kernel~$K$ has a slower decay at infinity, the result is no longer true. Lemma~\ref{rapidneclem} in Appendix~\ref{addcompapp} shows that, under this assumption, the functional~$\F$ is nowhere finite on the whole class of admissible functions~$\AAB$.

We also point out that this is the only part of the section in which we need the additional hypothesis~\eqref{Krapid} and future computations will involve neither~$\beta$, nor~$\bar{R}$, nor~$\Gamma$.

With the aid of the finiteness result yielded by Lemma~\ref{finfunlem}, we can now prove the existence of minimizers.

\begin{proposition} \label{perminprop}
There exists an absolute minimizer of the functional~$\F$ within the class~$\AAB$.
\end{proposition}
\begin{proof}
Our argumentation follows the lines of the standard Direct Method of the Calculus of Variations.

By Lemma~\ref{finfunlem} and the fact that~$\F$ is non-negative, we know that
$$
m := \inf \left\{ \F(u) : u \in \AAB \right\} \in [0, +\infty).
$$
Let then~$\{ u_j \}_{j \in \N} \subseteq \AAB$ be a minimizing sequence. Observe that we may assume without loss of generality that
\begin{equation} \label{ujle1}
|u_j| \le 1 \quad \mbox{a.e. in } \R^n,
\end{equation}
as this restriction only makes the energy~$\F$ decrease. Moreover, we fix an integer~$k > \max \{ -A, B \}$ and consider the Lipschitz domains
$$
\Omega_k := \tRn \cap \left\{ x \in \R^n : |\omega \cdot x| \le k  \right\}.
$$
By~\eqref{ujle1} and~\eqref{Kbounds} we have
\begin{align*}
[u_j]_{H^s(\Omega_k)}^2 & \le \int_{\Omega_k} \left( \int_{B_1(x)} \frac{\left| u_j(x) - u_j(y) \right|^2}{|x - y|^{n + 2 s}} \, dy \right) dx + 4 \int_{\Omega_k} \left( \int_{\R^n \setminus B_1(x)} \frac{dy}{|x - y|^{n + 2 s}} \right) dx \\
& \le \frac{2}{\lambda} \F(u_j) + \frac{2 n \alpha_n |\Omega_k|}{s},
\end{align*}
so that~$\{ u_j \}$ is bounded in~$H^s(\Omega_k)$, uniformly in~$j$. By the compact embedding of~$H^s(\Omega_k)$ into~$L^2(\Omega_k)$ (see e.g. Theorem~7.1 of~\cite{DPV12}), we then deduce that a subsequence of~$\{ u_j \}$ converges to some function~$u$ in~$L^2(\Omega_k)$ and, thus, a.e. in~$\Omega_k$. Using a diagonal argument (on~$j$ and~$k$), we may indeed find a subsequence~$\{ u_j^* \}$ of~$\{ u_j \}$ which converges to~$u$ a.e. in~$\tRn$. Furthermore, we may identify the~$u_j^*$'s and~$u$ with their~$\sim$-periodic extensions to~$\R^n$ and thus obtain that such convergence is a.e. in the whole space~$\R^n$. Accordingly,~$u \in \AAB$ and an application of Fatou's lemma shows that~$\F(u) = m$. This concludes the proof.
\end{proof}

\subsection{The minimal minimizer} \label{minminsubsec}

Denote by~$\MAB$ the set composed by the absolute minimizers of~$\F$ in~$\AAB$, i.e.
$$
\MAB := \bigg\{ u \in \AAB : \F(u) \le \F(v) \mbox{ for any } v \in \AAB \bigg\}.
$$
Clearly,~$\MAB$ is not empty, as shown by Proposition~\ref{perminprop}. Here below we introduce a particular element of the class~$\MAB$, that will turn out to be of central interest in the remainder of the paper.

\begin{definition} \label{minmindef}
We define the~\emph{minimal minimizer}~$\uAB$ as the infimum of~$\MAB$ as a subset of the partially ordered set~$( \AAB, \le )$. More specifically,~$\uAB$ is the unique function of~$\AAB$ for which
\begin{equation} \label{minminprop1}
\uAB \le u \mbox{ in } \R^n \mbox{ for every } u \in \MAB
\end{equation}
and
\begin{equation} \label{minminprop2}
\mbox{if } v \in \AAB \mbox{ is s.t. } v \le u \mbox{ in } \R^n \mbox{ for every } u \in \MAB, \mbox{ then } v \le \uAB \mbox{ in } \R^n.
\end{equation}
\end{definition}


Of course, the existence of the minimal minimizer is far from being established. Aim of the subsection is to prove that such function is in fact well-defined and that it belongs to~$\MAB$ itself.

\medskip

In order to construct~$\uAB$ we first need to show that the set~$\MAB$ is closed with respect to the operation of taking the minimum between two of its elements. To do this, we actually prove a stronger fact, which will be needed, in its full generality, only later in Subsection~\ref{birksubsec}.

\begin{lemma} \label{min2lem}
Let~$A \le A'$ and~$B \le B'$, with~$A < B$ and~$A' < B'$. If~$u \in \MAB$ and~$v \in \mathcal{M}_\omega^{A', B'}$, then~$\min \{ u, v \} \in \MAB$.
\end{lemma}
\begin{proof}
First, notice that~$\min \{ u, v \} \in \AAB$ and~$\max \{ u, v \} \in \mathcal{A}_\omega^{A', B'}$. Moreover, employing Lemma~\ref{uvmMlem} we deduce
$$
\F(\min \{ u, v \}) + \F(\max \{ u, v \}) \le \F(u) + \F(v).
$$
Taking advantage of this inequality, together with the fact that~$v \in \mathcal{M}_\omega^{A', B'}$, we get
$$
\F(\min \{ u, v \}) + \F(\max \{ u, v \}) \le \F(u) + \F(\max \{ u, v \}),
$$
which in turn implies that
$$
\F(\min \{ u, v \}) \le \F(u).
$$
Consequently,~$\min \{ u, v \} \in \MAB$.
\end{proof}

By choosing~$A = A'$ and~$B = B'$, we obtain the desired

\begin{corollary} \label{min2cor}
Let~$u, v \in \MAB$. Then,~$\min \{ u, v \} \in \MAB$.
\end{corollary}

Now that we know that the minimum between two - and, consequently, any finite number of - minimizers is still a minimizer, we can show that also the infimum over a \emph{countable} family of elements of~$\MAB$ belongs to~$\MAB$.

\begin{lemma} \label{minnumlem}
Let~$\{ u_j \}_{j \in \N}$ be a sequence of elements of~$\MAB$. Then,~$\inf_{j \in \N} \limits u_j \in \MAB$.
\end{lemma}
\begin{proof}
Write~$u_* := \inf_{j \in \N} \limits u_j$. We define inductively the auxiliary sequence
$$
v_j := \begin{cases}
u_1 & \mbox{if } j = 1 \\
\min \{ v_{j - 1}, u_j \} & \mbox{if } j \ge 2.
\end{cases}
$$
By Corollary~\ref{min2cor}, we know that~$\{ v_j \} \subseteq \MAB$. Moreover,~$v_j$ converges to~$u_*$ a.e. in~$\R^n$. An application of Fatou's lemma then yields that~$u_* \in \AAB$ and
$$
\F(u_*) \le \lim_{j \rightarrow +\infty} \F(v_j) = \F(v_k),
$$
for any~$k \in \N$. Therefore,~$u_* \in \MAB$.
\end{proof}

Finally, we are in position to prove the main result of the present subsection.

\begin{proposition} \label{minminismin}
The minimal minimizer~$\uAB$, as given by Definition~\ref{minmindef}, exists and belongs to~$\MAB$.
\end{proposition}
\begin{proof}
The set~$\MAB$ is separable with respect to convergence a.e., i.e. there exists a sequence~$\{ u_j \}_{j \in \N} \subseteq \MAB$ such that for any~$u \in \MAB$ we may pick a subsequence~$\{ u_{j_k }\}$ which converges to~$u$ a.e. in~$\R^n$. A rigorous proof of this fact can be found in Proposition~\ref{aesepprop} of Appendix~\ref{sepapp}. Set
$$
\uAB := \inf_{j \in \N} u_j.
$$
By Lemma~\ref{minnumlem}, we already know that~$\uAB \in \MAB$. We claim that~$\uAB$ is the minimal minimizer, i.e. that satisfies the properties~\eqref{minminprop1} and~\eqref{minminprop2} listed in Definition~\ref{minmindef}.

Take~$u \in \MAB$ and let~$\{ u_{j_k}\}$ be a subsequence of~$\{ u_j \}$ converging to~$u$ a.e. in~$\R^n$. By definition,~$\uAB \le u_{j_k}$ in~$\R^n$, for any~$k \in \N$. Hence, taking the limit as~$k \rightarrow +\infty$, condition~\eqref{minminprop1} follows.

Now we turn our attention to~\eqref{minminprop2} and we assume that there exists~$v \in \AAB$ such that~$v \le u$, for any~$u \in \M$. Then, in particular, we have~$v \le u_j$, for any~$j \in \N$ which implies~$v \le \uAB$. Thus,~\eqref{minminprop2} follows and the proof of the proposition is complete.
\end{proof}

\subsection{The doubling property} \label{doubsubsec}

An important feature of
the minimal minimizer is the so-called~\emph{doubling property} (or~\emph{no-symmetry-breaking property}). Namely, we prove in this subsection that~$\uAB$ is still the minimal minimizer with respect to functions having periodicity multiple of~$\sim$. In order to formulate precisely this result, we need a few more notation.

\medskip

Let~$z_1, \ldots, z_{n - 1} \in \Z^n$ denote some vectors spanning the~$(n - 1)$-dimensional lattice induced by~$\sim$. Thus, any~$k \in \Z^n$ such that~$\omega \cdot k = 0$ may be written as
$$
k = \sum_{i = 1}^{n - 1} \mu_i z_i,
$$
for some~$\mu_1, \ldots, \mu_{n - 1} \in \Z$. For a fixed~$m \in \N^{n - 1}$, we introduce the equivalence relation~$\sim_m$, defined by setting
$$
x \sim_m y \quad \mbox{if and only if} \quad x - y = \sum_{i = 1}^{n - 1} \mu_i m_i z_i, \mbox{ for some } \mu_1, \ldots, \mu_{n - 1} \in \Z.
$$
Also, set~$\tRnm := \R^n / \sim_m$ and denote by~$L^2_\loc(\tRnm)$ the space of~$\sim_m$-periodic functions which belong to~$L^2_\loc(\R^n)$. Note that~$\tRnm$ contains exactly~$m_1 \cdot \ldots \cdot m_{n - 1}$ copies of~$\tRn$. Indeed, the relation~$\sim_m$ is weaker than~$\sim$ and~$L^2_\loc(\tRn) \subseteq L^2_\loc(\tRnm)$. We consider the space of admissible functions
$$
\Am := \bigg\{ u \in L^2_\loc(\tRnm) : u(x) \ge \frac{9}{10} \mbox{ if } \omega \cdot x \le A \mbox{ and } u(x) \le - \frac{9}{10} \mbox{ if } \omega \cdot x \ge B \bigg\},
$$
related to this new equivalence relation, together with the set of absolute minimizers
$$
\Mm := \bigg\{ u \in \Am : \Fm(u) \le \Fm(v) \mbox{ for any } v \in \Am \bigg\},
$$
of the functional
\begin{align*}
\Fm(u) := & \, \K(u; \tRnm, \R^n) + \P(u; \tRnm) \\
= & \, \frac{1}{2} \int_{\tRnm} \int_{\R^n} |u(x) - u(y)|^2 K(x, y) \, dx dy + \int_{\tRnm} W(x, u(x)) \, dx.
\end{align*}
We indicate with~$\um$ the minimal minimizer of the class~$\Mm$. Of course, its existence is granted by the same arguments of Subsection~\ref{minminsubsec}.

Finally, given a function~$u: \R^n \to \R$ and a vector~$z \in \R^n$, we denote the translation of~$u$ in the direction~$z$ as
\begin{equation} \label{functrnasl}
\tau_z u(x) := u(x - z) \quad \mbox{for any } x \in \R^n.
\end{equation}

\medskip

After this preliminary work, we can now prove that the minimal minimizer
in a class of larger period coincides with the one in a class of smaller
period:

\begin{proposition} \label{u=umprop}
For any~$m \in \N^{n - 1}$, it holds~$\um = \uAB$.
\end{proposition}
\begin{proof}
For simplicity of exposition we restrict ourselves to the case in which~$m_1 = 2$ and~$m_i = 1$, for every~$i = 2, \ldots, n - 1$. The approach in the general case would be analogous, but much heavier in notation.

We begin by showing that~$\um \le \uAB$. Notice that the inequality follows if we prove that~$\uAB \in \Mm$. To see this, we consider the translation~$\tau_{z_1} \um$ of~$\um$ in the~\emph{doubled} direction~$z_1$. Clearly,~$\tau_{z_1} \um \in \Mm$. Then, we define
$$
\hatum := \min \left\{ \um, \tau_{z_1} \um \right\}.
$$
Observe that~$\hatum$ is~$\sim$-periodic and hence belongs to~$\AAB$. Then,
$$
\Fm(\uAB) = 2 \F(\uAB) \le 2 \F(\hatum) = \Fm(\hatum) \le \Fm(\um),
$$
where the last inequality follows by Lemma~\ref{uvmMlem}, arguing as in the proof of Lemma~\ref{min2lem}. Accordingly, we deduce that~$\uAB \in \Mm$ and so~$\um \le \uAB$, since~$\um$ is the minimal minimizer of~$\Mm$.

On the other hand, being~$\hatum \in \Mm$ and~$\uAB \in \Am$, we have
$$
\F(\hatum) = \frac{1}{2} \Fm(\hatum) \le \frac{1}{2} \Fm(\uAB) = \F(\uAB),
$$
which implies that~$\hatum \in \MAB$. Consequently,~$\uAB \le \hatum \le \um$, and the proposition is therefore proved.
\end{proof}

\subsection{Minimization with respect to compact perturbations} \label{cptsubsec}

In the previous subsections we have been concerned with functionals of the type~$\Fm$. We proved that absolute minimizers for such functionals exist in particular classes of~$\sim_m$-periodic functions. Since our ultimate goal is the construction of class~A minimizers for the energy~$\E$, we now need to show that the elements of~$\MAB$ are also minimizers of~$\E$ with respect to compact perturbations occurring within the strip
\begin{equation} \label{stripAB}
\SAB := \left\{ x \in \R^n : \omega \cdot x \in [A, B] \right\}.
\end{equation}
In what follows, it will also be useful to introduce the quotient
\begin{equation} \label{quotABm}
\tSm := \SAB / \sim_m.
\end{equation}

\medskip

The first result of the subsection addresses a general relationship intervening between the two functionals~$\E$ and~$\Fm$.

\begin{lemma} \label{EFrellem}
Let~$u \in \Am$ be a bounded function with finite~$\Fm$ energy. Given an open set~$\Omega$ compactly contained in~$\tSm$,\footnote{We stress that here~$\Omega$ is meant to be compactly contained in a fundamental domain of~$\tSm$, and not only in the quotient set itself. The difference is that we do not allow~$\Omega$ to touch the \emph{lateral} boundary of the domain. \label{cptfoot}} let~$v$ be another bounded function such that~$u = v$ outside~$\Omega$ and set~$\varphi := v - u$. Denoting with~$\tilde{v}$ and~$\tilde{\varphi}$ the~$\sim_m$-periodic extensions to~$\R^n$ of~$v|_{\tRnm}$ and~$\varphi|_{\tRnm}$, respectively, it then holds
\begin{equation} \label{EFrel}
\E(v; \tRnm) - \E(u; \tRnm) = \Fm(\tilde{v}) - \Fm(u) + \int_{\tRnm} \int_{\R^n \setminus \tRnm} \tilde{\varphi}(x) \tilde{\varphi}(y) K(x, y) \, dx dy.
\end{equation}
In particular, if~$u \in \Mm$, then
\begin{equation} \label{uphiE}
\E(v; \tRnm) - \E(u; \tRnm) \ge \int_{\tRnm} \int_{\R^n \setminus \tRnm} \tilde{\varphi}(x) \tilde{\varphi}(y) K(x, y) \, dx dy.
\end{equation}
\end{lemma}

Note that the integral written on the right-hand sides of~\eqref{EFrel} and~\eqref{uphiE} is finite, since~$\varphi$ is compactly supported on~$\tSm$ and bounded. For a justification of this fact, see Lemma~\ref{intfinlem} in Appendix~\ref{addcompapp}.

\begin{proof}[Proof of Lemma~\ref{EFrellem}]
For simplicity, we restrict ourselves to consider~$m = (1, \ldots, 1)$, the general case being completely analogous. Moreover, it is enough to prove formula~\eqref{EFrel}, as~\eqref{uphiE} then easily follows by noticing that~$\tilde{v} \in \Am$.

Recalling definition~\eqref{Edef}, we first inspect the term~$\K(v; \tRn, \R^n \setminus \tRn)$. To this aim, let~$x \in \tRn$ and~$y \in \R^n \setminus \tRn$. We compute
\begin{align*}
\left| v(x) - v(y) \right|^2 & = \left| u(x) + \varphi(x) - u(y) \right|^2 \\
& = \left| u(x) + \tilde{\varphi}(x) - u(y) - \tilde{\varphi}(y) \right|^2 + 2 \tilde{\varphi}(y) \left( u(x) + \tilde{\varphi}(x) - u(y) \right) - \tilde{\varphi}(y)^2 \\
& = \left| \tilde{v}(x) - \tilde{v}(y) \right|^2 + \left| u(x) - u(y) \right|^2 - \left| u(x) - u(y) - \tilde{\varphi}(y) \right|^2 + 2 \tilde{\varphi}(x) \tilde{\varphi}(y),
\end{align*}
and thus
\begin{equation} \label{EFtech}
\begin{aligned}
\K(v; \tRn, \R^n \setminus \tRn) & = \K(\tilde{v}, \tRn, \R^n \setminus \tRn) + \K(u; \tRn, \R^n \setminus \tRn) \\
& \quad - \frac{1}{2} \int_{\tRn} \left( \int_{\R^n \setminus \tRn} \left| u(x) - u(y) - \tilde{\varphi}(y) \right|^2 K(x, y) \, dy \right) dx \\
& \quad + \int_{\tRn} \int_{\R^n \setminus \tRn} \tilde{\varphi}(x) \tilde{\varphi}(y) K(x, y) \, dx dy.
\end{aligned}
\end{equation}
Notice now that
$$
\R^n \setminus \tRn = \bigcup_{\substack{k \in \Z^n \setminus \{ 0 \} \\ \omega \cdot k = 0}} \left( \tRn + k \right),
$$
so that we may write the integral on the second line of~\eqref{EFtech} as
$$
\sum_{\substack{k \in \Z^n \setminus \{ 0 \} \\ \omega \cdot k = 0}} \int_{\tRn} \left( \int_{\tRn + k} \left| u(x) - u(y) - \tilde{\varphi }(y) \right|^2 K(x, y) \, dy \right) dx.
$$
By changing variables as~$w := x - k$,~$z := y - k$, recalling~\eqref{Kper} and taking advantage of the periodicity of~$u$ and~$\tilde{\varphi}$, we find that
\begin{align*}
& \int_{\tRn} \left( \int_{\tRn + k} \left| u(x) - u(y) - \tilde{\varphi}(y) \right|^2 K(x, y) \, dy \right) dx \\
& \hspace{80pt} = \int_{\tRn - k} \left( \int_{\tRn} \left| u(w) - u(z) - \tilde{\varphi}(z) \right|^2 K(w, z) \, dz \right) dw \\
& \hspace{80pt} = \int_{\tRn - k} \left( \int_{\tRn} \left| v(w) - v(z) \right|^2 K(w, z) \, dz \right) dw.
\end{align*}
By summing up on~$k$ this identity,~\eqref{EFtech} becomes
\begin{align*}
\K(v; \tRn, \R^n \setminus \tRn) & = \K(\tilde{v}, \tRn, \R^n \setminus \tRn) + \K(u; \tRn, \R^n \setminus \tRn) - \K(v; \R^n \setminus \tRn, \tRn) \\
& \quad + \int_{\tRn} \int_{\R^n \setminus \tRn} \tilde{\varphi}(x) \tilde{\varphi}(y) K(x, y) \, dx dy.
\end{align*}
The thesis then follows by noticing that
$$
\K(v; \tRn, \tRn) = \K(\tilde{v}; \tRn, \tRn) \quad \mbox{and} \quad \P(v; \tRn) = \P(\tilde{v}; \tRn),
$$
and recalling the definitions of~$\E$ and~$\F$.
\end{proof}

With this in hand, we may state the following proposition, where we prove that the absolute minimizers of~$\Fm$ in the class~$\Am$ also minimizes~$\E$ with respect to compact perturbations occurring inside~$\tSm$.

\begin{proposition} \label{cptminprop}
Let~$u \in \Mm$. Then,~$u$ is a local minimizer of~$\E$ in every open set~$\Omega$ compactly contained in~$\tSm$, that is
\begin{equation}\label{TAR89}
\E(u; \Omega) \le \E(v; \Omega),
\end{equation}
for any~$v$ which coincides with~$u$ outside~$\Omega$.
\end{proposition}
\begin{proof}
First of all, we assume without loss of generality that~$\E(v; \Omega) < +\infty$ and~$|v| \le 1$ a.e. in~$\R^n$. Set~$\varphi := v - u$ and observe that~$\varphi$ is supported on~$\Omega$. We will show that
inequality~\eqref{TAR89}
holds on the larger region~$\tRnm$, in place of~$\Omega$, i.e.
\begin{equation} \label{ulocminE}
\E(u; \tRnm) \le \E(v; \tRnm).
\end{equation}
This will imply~\eqref{TAR89}, in light of Remark~\ref{restrmk}.

To prove~\eqref{ulocminE}, we first
notice that if~$\varphi$ is either non-negative or non-positive, then~\eqref{ulocminE} follows as a direct consequence of inequality~\eqref{uphiE}. On the other hand, if~$\varphi$ is sign-changing, we consider the minimum and the maximum between~$u$ and~$u + \varphi$. Recalling Lemma~\ref{uvmMlem} it is immediate to see that
$$
\E(\min \{ u, u + \varphi \}; \tRnm) + \E(\max \{ u, u + \varphi \}; \tRnm) \le \E(u; \tRnm) + \E(u + \varphi; \tRnm).
$$
Moreover, since it holds
$$
\min \{ u, u + \varphi \} = u - \varphi_- \quad \mbox{and} \quad \max \{ u, u + \varphi \} = u + \varphi_+,
$$
we may apply~\eqref{uphiE} and get
\begin{align*}
2 \, \E(u; \tRnm) & \le \E(u - \varphi_-; \tRnm) + \E(u + \varphi_+; \tRnm) \\
& = \E(\min \{ u, u + \varphi \}; \tRnm) + \E(\max \{ u, u + \varphi \}; \tRnm) \\
& \le \E(u; \tRnm) + \E(u + \varphi; \tRnm).
\end{align*}
This leads to~\eqref{ulocminE}.
\end{proof}

From this proposition and the results of Subsection~\ref{doubsubsec}, we immediately deduce the following

\begin{corollary} \label{minmincptcor}
The minimal minimizer~$\uAB$ is a local minimizer of~$\E$ in every bounded open set~$\Omega$ compactly contained in the strip~$\SAB$.
\end{corollary}
\begin{proof}
Given~$\Omega$, we take~$m \in \N^{n - 1}$ large enough in order to have~$\Omega \subset \subset \tSm$. In view of Proposition~\ref{u=umprop},~$\uAB$ is the minimal minimizer with respect to~$\Mm$. But then, by Proposition~\ref{cptminprop},~$\uAB$ is a local minimizer of~$\E$ in~$\Omega$.
\end{proof}

\subsection{The Birkhoff property} \label{birksubsec}

In this subsection we
introduce an interesting geometric feature shared by the level sets of the minimal minimizer: the~\emph{Birkhoff property}
(also known in the literature as ``non-self-intersection property''). Namely,
the level sets of the minimal minimizers are ordered under translations.

\medskip

In order to give a formal definition of this property, the following notation will be useful.

Similarly to what we did in~\eqref{functrnasl} for functions, we consider the translation of a set~$E \subseteq \R^n$ with respect to a vector~$z \in \R^n$
\begin{equation} \label{settransl}
\tau_z E := E + z = \left\{ x + z : x \in E \right\}.
\end{equation}
Notice that, with this notation, the translation of a sublevel set then is given by
\begin{equation} \label{sublevtransl}
\tau_z \{ u < \theta \} = \{ \tau_z u < \theta \},
\end{equation}
and analogously for the superlevel sets.

\begin{definition}
Let~$E$ be a subset of~$\R^n$. We say that~$E$ has the~\emph{Birkhoff property} with respect to a vector~$\varpi \in \R^n$ if:
\begin{enumerate}[$\bullet$]
\item $\tau_k E \subseteq E$, for any~$k \in \Z^n$ such that~$\varpi \cdot k \le 0$, and
\item $\tau_k E \supseteq E$, for any~$k \in \Z^n$ such that~$\varpi \cdot k \ge 0$.
\end{enumerate}
\end{definition}

Before exploring the connection between the minimal minimizer and the Birkhoff property, we present a proposition which addresses Birkhoff sets from an abstract point of view and displays a rigidity feature of those of such sets that have \emph{fat} interior.

\begin{proposition} \label{fatbirk}
Let~$E \subseteq \R^n$ be a set satisfying the Birkhoff property with respect to a vector~$\varpi \in \R^n \setminus \{ 0 \}$. If~$E$ contains a ball of radius~$\sqrt{n}$, then it also contains a half-space which includes the center of the ball, has delimiting hyperplane orthogonal to~$\varpi$ and is such that~$\varpi$ points outside of it.
\end{proposition}
\begin{proof}
Let~$B_{\sqrt{n}}(x_0)$ be the ball of radius~$\sqrt{n}$ and center~$x_0$ contained in~$E$. By the Birkhoff property, it holds
$$
\bigcup_{\substack{k \in \Z^n \\ \varpi \cdot k \le 0}} \tau_k B_{\sqrt{n}}(x_0) \subseteq \bigcup_{\substack{k \in \Z^n \\ \varpi \cdot k \le 0}} \tau_k E \subseteq E.
$$
The thesis now follows by observing that the set on the left-hand side above contains the half-space~$\{ \varpi \cdot (x - x_0) < \varepsilon \}$, for some~$\varepsilon > 0$.
\end{proof}

Now we show that the level sets of the minimal minimizer are Birkhoff sets. Recalling the relation between translations and level sets established in~\eqref{sublevtransl}, we have

\begin{proposition} \label{ubirk}
Let~$\theta \in \R$. Then, the superlevel set~$\left\{ \uAB > \theta \right\}$ has the Birkhoff property with respect to~$\omega$. Explicitly,
\begin{enumerate}[$\bullet$]
\item $\left\{ \tau_k \uAB > \theta \right\} \subseteq \left\{ \uAB > \theta \right\}$, for any~$k \in \Z^n$ such that~$\omega \cdot k \le 0$, and
\item $\left\{ \tau_k \uAB > \theta \right\} \supseteq \left\{ \uAB > \theta \right\}$, for any~$k \in \Z^n$ such that~$\omega \cdot k \ge 0$.
\end{enumerate}
Analogously, the sublevel set~$\{ \uAB < \theta \}$ has the Birkhoff property with respect to~$-\omega$.\\
The same statements still hold if we replace strict level sets with broad ones.
\end{proposition}
\begin{proof}
Let~$v := \min \{ \uAB, \tau_k \uAB \}$ and observe that~$\tau_k \uAB$ is the minimal minimizer with respect to the strip~$\tau_k \SAB = \mathcal{S}_\omega^{A + \omega \cdot k, B + \omega \cdot k}$. If~$\omega \cdot k \le 0$ then by Lemma~\ref{min2lem} it follows that~$v \in \mathcal{M}_\omega^{A + \omega \cdot k, B + \omega \cdot k}$. Thus,~$\tau_k \uAB \le v \le \uAB$ and hence
$$
\left\{ \tau_k \uAB > \theta \right\} \subseteq \left\{ \uAB > \theta \right\}.
$$
On the other hand, if~$\omega \cdot k \ge 0$ then~$v \in \MAB$ and therefore
$$
\left\{ \uAB < \theta \right\} \subseteq \left\{ \tau_k \uAB < \theta \right\}.
$$

The conclusion for the sublevel set~$\{ \uAB \le \theta \}$ follows observing that a set~$E \subseteq \R^n$ is Birkhoff with respect to a vector~$\varpi \in \R^n$ if and only if~$\R^n \setminus E$ is Birkhoff with respect to~$- \varpi$.

Finally, by writing
$$
\left\{ \uAB < \theta \right\} = \bigcup_{k \in \N} \left\{ \uAB \le \theta - 1 / k \right\},
$$
and noticing that the union of a family of sets that are Birkhoff with respect to a mutual vector is itself Birkhoff with respect to the same vector, we deduce that~$\{ \uAB < \theta \}$ has the Birkhoff property with respect to~$- \omega$. In a similar way one checks that the superlevel set~$\{ \uAB \ge \theta \}$ is Birkhoff with respect to~$\omega$.
\end{proof}

\subsection{Unconstrained and class~A minimization} \label{unconstrsubsec}

From now on we mainly restrict our attention to strips of the form
$$
\S := \mathcal{S}_\omega^{0, M} = \left\{ x \in \R^n : \omega \cdot x \in [0, M] \right\}.
$$
We simply write~$\A$ for the space~$\mathcal{A}_\omega^{0, M}$ of admissible functions,~$\M$ for the absolute minimizers and~$\u$ for the minimal minimizer. We also assume~$M > 10 |\omega|$, in order to avoid degeneracies caused by too narrow strips.

The main purpose of this subsection is to show that the minimal minimizer~$\u$ becomes unconstrained for large, universal values of~$M / |\omega|$. By \emph{unconstrained} we mean that~$\u$ no longer \emph{feels} the boundary data prescribed outside the strip~$\S$ and gains additional minimizing properties in the whole space~$\R^n$. Of course, we will be more precise on this later in Proposition~\ref{unconstrprop}.

\medskip

We begin by adapting the results of Sections~\ref{regsec} and~\ref{enestsec} to the minimal minimizer~$\u$. Recall that~$\u$ is a local minimizer for~$\E$ inside the strip~$\S$, thanks to Corollary~\ref{minmincptcor}.

In view of Corollary~\ref{minminholdercor}, we deduce that there exist universal quantities~$\alpha \in (0, 1)$ and~$C_1 \ge 1$ for which
\begin{equation} \label{holderunipre}
\| \u \|_{C^{0, \alpha}(S)} \le C_1,
\end{equation}
for any open set~$S \subset \subset \S$ such that~$\dist \left(S, \partial \S \right) \ge 1$.

On the other hand, Proposition~\ref{enestprop} tells that, given~$x_0 \in \S$ and~$R \ge 3$ in such a way that~$B_{R + 2}(x_0) \subset \subset \S$, it holds
\begin{equation} \label{minminenest}
\E(\u; B_R(x_0)) \le C_2 R^{n - 1} \Psi_s(R),
\end{equation}
for a universal constant~$C_2 > 0$. Recall that~$\Psi_s(R)$ was defined in~\eqref{Psidef}.

\medskip

Now that~\eqref{holderunipre} and~\eqref{minminenest} are established, we may proceed to the core proposition of the present subsection.

\begin{proposition} \label{dist1prop}
There exists a universal~$M_0 > 0$ such that if~$M \ge M_0 |\omega|$, then the superlevel set~$\{ \u > - 9 / 10 \}$ is at least at distance~$1$ from the upper constraint~$\{ \omega \cdot x = M \}$ delimiting~$\S$.
\end{proposition}
\begin{proof}
In the course of this proof we will often indicate balls and cubes without any explicit mention of their center. Thus,~$B$ will be for instance used to denote a ball not necessarily centered at the origin, in contrast with the notation adopted in the rest of the paper.

We claim that
\begin{equation} \label{claimuni}
\begin{aligned}
& \mbox{there exists a universal constant } M_0 \ge 8 n \mbox{ such that, for any } M \ge M_0 |\omega|, \\
& \mbox{we can find a ball } B_{\sqrt{n}}(\bar{z}) \subset \subset \S, \mbox{ for some } \bar{z} \in \S, \mbox{ on which} \\
& \mbox{either } \u \ge 9 / 10 \mbox{ or } \u \le - 9 / 10.
\end{aligned}
\end{equation}

Let~$M \ge 8 n |\omega|$ be given and suppose that for any ball~$\tilde{B}$ of radius~$\sqrt{n}$ compactly contained in~$\S$, there exists a point~$\tilde{x} \in \tilde{B}$ such that~$|\u(\tilde{x})| < 9 / 10$. If we show that~$M / |\omega|$ is less or equal to a universal value~$M_0$, claim~\eqref{claimuni} would then be true.

Let~$k \ge 2$ be the only integer for which
\begin{equation} \label{relMk}
k \le \frac{M}{4 n |\omega|} < k + 1.
\end{equation}
Take a point~$x_0 \in \S$ lying on the hyperplane~$\{ \omega \cdot x = M/ 2 \}$ and consider the ball~$B = B_{n k}(x_0)$. By~\eqref{relMk}, we have that~$B \subset \subset \S$, with
\begin{equation} \label{distge2}
\dist \left( B, \partial \S \right) = \frac{M}{2 |\omega|} - n k \ge n k \ge 4.
\end{equation}
Consequently, we may apply the bound in~\eqref{holderunipre} to deduce that
\begin{equation} \label{holderuni}
\| \u \|_{C^{0, \alpha}(B)} \le C_1.
\end{equation}

Let now~$Q$ be a cube of sides~$2 \sqrt{n} k$, centered at~$x_0$. Of course,~$Q \subset B$. It is easy to see that~$Q$ may be partitioned (up to a negligible set) into a collection~$\{ Q_j \}_{j = 1}^{k^n}$ of cubes with sides of length~$2 \sqrt{n}$, parallel to those of~$Q$. Moreover, we denote with~$B_j \subset Q_j$ the ball of radius~$\sqrt{n}$ having the same center of~$Q_j$. See Figure~\ref{cubepart}.

\begin{figure}[h]
\centering
\fbox{\includegraphics[width=0.9\textwidth]{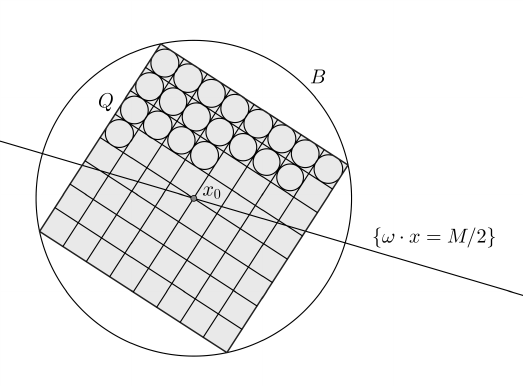}}
\caption{The partition of the cube~$Q$ into the subcubes~$Q_j$'s and the concentric balls~$B_j$'s.}
\label{cubepart}
\end{figure}

In view of our starting assumption, for any~$j = 1, \ldots, k^n$ there exists a point~$\tilde{x}_j \in B_j$ at which~$|\u(\tilde{x}_j)| < 9 / 10$. We claim that
\begin{equation} \label{99/100}
|\u| < 99 / 100 \quad \mbox{in } B_{r_0}(\tilde{x}_j),
\end{equation}
for some universal radius~$r_0 \in (0, 1)$. Indeed, setting~$r_0 := \left( 9 / (100 C_1) \right)^{1 / \alpha}$, by~\eqref{holderuni} we get
$$
|\u(x)| \le |\u(\tilde{x}_j)| + C_1 |x - \tilde{x}_j|^\alpha < \frac{9}{10} + C_1 r_0^\alpha = \frac{99}{100},
$$
for any~$x \in B_{r_0}(\tilde{x}_j)$. Hence,~\eqref{99/100} is established. Furthermore, since~$\tilde{x}_j \in B_j \subset Q_j$, we have
\begin{equation} \label{fraction}
\left| B_{r_0}(\tilde{x}_j) \cap Q_j \right| \ge \frac{1}{2^n} \left| B_{r_0}(\tilde{x}_j) \right| = \frac{\alpha_n}{2^n} r_0^n.
\end{equation}

By combining~\eqref{99/100} and~\eqref{fraction}, recalling~\eqref{Wgamma} we compute
\begin{align*}
\P \left( \u; B \right) & \ge \P \left( \u; Q \right) = \sum_{j = 1}^{k^n} \P \left( \u; Q_j \right) \\
& \ge \sum_{j = 1}^{k^n} \P \left( \u; B_{r_0}(\tilde{x}_j) \cap Q_j \right) = \sum_{j = 1}^{k^n} \int_{B_{r_0}(\tilde{x}_j) \cap Q_j} W \left( x, \u(x) \right) \, dx \\
& \ge \gamma \left( \frac{99}{100} \right) \sum_{j = 1}^{k^n} |B_{r_0}(\tilde{x}_j) \cap Q_j| \ge \frac{\alpha_n}{2^n} r_0^n \gamma \left( \frac{99}{100} \right) k^n \\
& =: C_3 k^n,
\end{align*}
with~$C_3 > 0$ universal. On the other hand,~\eqref{minminenest} implies that
$$
\P(\u; B) \le \E(\u; B) \le C_2 \left( n k \right)^{n - 1} \Psi_s \left( n k \right) \le C_4 k^{n - 1} \Psi_s(k),
$$
for some universal~$C_4 > 0$. Note that the energy estimate~\eqref{minminenest} may be applied to the ball~$B$ thanks to~\eqref{distge2}. Comparing the last two inequalities and recalling~\eqref{Psidef}, we find out that~$k$ cannot be greater than a universal constant. By~\eqref{relMk}, the same holds true for the quotient~$M / |\omega|$ and hence~\eqref{claimuni} follows.

Now, we want to rule out the possibility of~$\u$ being greater or equal to~$9 / 10$ on~$B_{\sqrt{n}}(\bar{z})$, thus showing that~$\u \le - 9 / 10$ in~$B_{\sqrt{n}}(\bar{z})$. By contradiction, assume that
\begin{equation} \label{contrhyp}
\u \ge 9 / 10 \quad \mbox{in } B_{\sqrt{n}}(\bar{z}).
\end{equation}
In view of Proposition~\ref{ubirk} the set~$\left\{ \u \ge 9 / 10 \right\}$ has the Birkhoff property with respect to~$\omega$. Hence, thanks to~\eqref{contrhyp} and Proposition~\ref{fatbirk}, this superlevel set contains the half-space~$\Pi_- := \{ \omega \cdot (x - \bar{z}) < 0 \}$. Since~$B_{\sqrt{n}}(\bar{z}) \subset \S$, we then deduce that the distance of~$\partial \Pi_-$ from the lower constraint~$\{ \omega \cdot x = 0 \}$ is at least~$1$. Accordingly, if we assume without loss of generality that~$\omega_1 > 0$, then the translation~$\tau_{- e_1} \u$ belongs to~$\A$ (recall definition~\eqref{functrnasl}). But then, the periodicity assumptions~\eqref{Kper}-\eqref{Wper} imply that~$\F(\tau_{- e_1} \u) = \F(\u)$ and thus~$\tau_{- e_1} \u \in \M$. Being~$\u$ the minimal minimizer, we conclude that
$$
\u(x + e_1) = \tau_{- e_1} \u(x) \ge \u(x) \quad \mbox{for a.e. } x \in \R^n. 
$$
By iterating this inequality we then find that
$$
\u(x + \ell e_1) \ge \u(x) \ge \frac{9}{10} \quad \mbox{for a.e. } x \in \Pi_- \mbox{ and any } \ell \in \N,
$$
i.e.,~$\u \ge 9 / 10$ a.e. in~$\R^n$, in contradiction with the fact that, by construction,~$\u \le - 9 / 10$ in~$\{ \omega \cdot x \ge M \}$.

As a result,~$\u \le - 9 / 10$ on the ball~$B_{\sqrt{n}}(\bar{z})$. The proof then finishes by applying once again Propositions~\ref{ubirk} and~\ref{fatbirk} to the sublevel set~$\left\{ \u \le - 9 / 10 \right\}$.
\end{proof}

\begin{corollary} \label{upperunconst}
If~$M \ge M_0 |\omega|$, then~$\u = u_\omega^{M + a}$, for any~$a \ge 0$.
\end{corollary}
\begin{proof}
Fix~$M \ge M_0 |\omega|$ and~$a \in [0, 1]$. By applying Proposition~\ref{dist1prop} to the minimal minimizer~$u_\omega^{M + a}$, we find that~$u_\omega^{M + a} \le - 9 / 10$ a.e. in the half-space~$\{ \omega \cdot x \ge M \}$. Hence,~$u_\omega^{M + a} \in \A$ and~$\F(\u) \le \F(u_\omega^{M + a})$, by the minimization properties of~$\u$. On the other hand, clearly~$\u \in \mathcal{A}_\omega^{M + a}$, so that we also have~$\F(u_\omega^{M + a}) \le \F(\u)$. Thus, both~$\u$ and~$u_\omega^{M + a}$ belong to~$\M \cap \mathcal{M}_\omega^{M + a}$ and, consequently, they define the same function.

By iteration, the arguments extends to any~$a \ge 0$.
\end{proof}

This corollary essentially tells that when~$M / |\omega|$ is greater than the universal constant~$M_0$ found in Proposition~\ref{dist1prop}, then the upper constraint~$\{ \omega \cdot x = M \}$ becomes immaterial for the minimal minimizer~$\u$, which starts attaining values below the threshold~$- 9 / 10$ well before touching that constraint.

The next result shows that a similar behavior also occurs with the lower constraint~$\{ \omega \cdot x = 0 \}$, thus proving that the minimal minimizer is unconstrained. Recalling the notation introduced right above Lemma~\ref{min2lem}, we state the following

\begin{proposition} \label{unconstrprop}
If~$M \ge M_0 |\omega|$, then~$\u$ is unconstrained, that is~$\u \in \mathcal{M}_\omega^{- a, M + a}$, for any~$a \ge 0$.
\end{proposition}
\begin{proof}
Let~$k \in \Z^n$ be such that~$\omega \cdot k \ge a$. Given~$v \in \mathcal{A}_\omega^{- a, M + a}$, we consider its translation~$\tau_k v \in \mathcal{A}_\omega^{M + a + \omega \cdot k}$. By Corollary~\ref{upperunconst}, it then holds~$\F(\u) \le \F(\tau_k v)$. The thesis then follows, as~$\F(v) = \F(\tau_k v)$ by~\eqref{Kper}-\eqref{Wper}.
\end{proof}

To conclude the subsection, we combine the previous proposition with the results of Subsection~\ref{cptsubsec} and obtain that~$\u$ is indeed a class~A minimizer.

\begin{theorem}
If~$M \ge M_0 |\omega|$, then~$\u$ is a class~A minimizer of the functional~$\E$.
\end{theorem}
\begin{proof}
Let~$\Omega$ be any given bounded subset of~$\R^n$. Take~$a \ge 0$  and~$m \in \Z^{n - 1}$ large enough to have~$\Omega$ compactly contained in the quotient $\widetilde{\mathcal{S}}_{\omega, m}^{- a, M + a}$ (recall notation~\eqref{quotABm}). By virtue of Proposition~\ref{u=umprop} we know that~$u_\omega^{- a, M + a}$ is the minimal minimizer of the class~$\mathcal{M}_{\omega, m}^{- a, M + a}$. On the other hand, Proposition~\ref{unconstrprop} yields~$\F(\u) = \F(u_\omega^{- a, M + a})$. Recalling the terminology introduced in Subsection~\ref{doubsubsec}, we then have
$$
\Fm(\u) = c_m \F(\u) = c_m \F(u_\omega^{- a, M + a}) = \Fm(u_\omega^{- a, M + a}),
$$
with~$c_m = \prod_{i = 1}^{n - 1} m_i$. Hence,~$\u \in \mathcal{M}_{\omega, m}^{- a, M + a}$ and Proposition~\ref{cptminprop} implies that~$\u$ is a local minimizer of~$\E$ in~$\Omega$.
\end{proof}

\subsection{The case of irrational directions} \label{irrsubsec}

Here we finish the proof of Theorem~\ref{mainthm} for kernels satisfying hypothesis~\eqref{Krapid}, by extending the results obtained in the previous subsections to irrational vectors~$\omega$. This task is accomplished by means of an approximation argument, whose most technical steps are inspired by~\cite[Section~7]{BV08}.

\medskip

Fix~$\omega \in \R^n \setminus \Q^n$ and consider a sequence~$\{ \omega_j \}_{j \in \N} \subset \Q^n \setminus \{ 0 \}$ converging to~$\omega$. Denote with~$u_j$ the class~A minimizer corresponding to~$\omega_j$, given by our construction. We recall that~$u_j \in H^s_\loc(\R^n) \cap L^\infty(\R^n)$, with~$|u_j| \le 1$ in~$\R^n$, and that
\begin{equation} \label{ujplanelike}
\left\{ x \in \R^n : |u_j(x)| \le \frac{9}{10} \right\} \subseteq \left\{ x \in \R^n : \frac{\omega_j}{|\omega_j|} \cdot x \in [0, M_0] \right\},
\end{equation}
for any~$j \in \N$. Moreover, by Corollary~\ref{minminholdercor}, the~$u_j$'s are uniformly bounded in~$C^{0, \alpha}(\R^n)$, for some universal~$\alpha \in (0, 1)$. Hence, by Arzel\`a-Ascoli Theorem there exists a subsequence of~$\{ u_j \}$ - which, without loss of generality, we will assume to be~$\{ u_j \}$ itself - converging to some continuous function~$u$, uniformly on compact subsets of~$\R^n$.

Of course,~$|u| \le 1$ in~$\R^n$. Also,~\eqref{ujplanelike} passes to the limit, so that the same inclusion holds with~$u$ and~$\omega$ replacing~$u_j$ and~$\omega_j$. In order to finish the proof of Theorem~\ref{mainthm} we therefore only need to show that~$u$ is a class~A minimizer of~$\E$. To do this, let~$R \ge 1$ be a fixed number: we claim that~$u$ is a local minimizer of~$\E$ in~$B_R$, that is~$\E(u; B_R) < +\infty$ and
\begin{equation} \label{uclassA}
\E(u; B_R) \le \E(u + \varphi; B_R) \quad \mbox{for any } \varphi \mbox{ supported inside } B_R.
\end{equation}
Observe that, going back to Remark~\ref{restrmk}, this implies that~$u$ is a class~A minimizer.

To see that~\eqref{uclassA} is true, we first apply Proposition~\ref{enestprop} to~$u_j$ and obtain that
\begin{equation} \label{Eu_j}
\E(u_j; B_{R + 1}) \le C_R,
\end{equation}
for some constant~$C_R > 0$ independent of~$j$. Furthermore, by Fatou's lemma, we know that
\begin{equation} \label{uuj}
\E(u; B_{R + \tau}) \le \liminf_{j \rightarrow +\infty} \E(u_j; B_{R + \tau}),
\end{equation}
for any~$\tau \in [0, 1]$, and thus, in particular,
\begin{equation} \label{Eu}
\E(u; B_R) \le \E(u; B_{R + 1}) \le C_R < +\infty.
\end{equation}
Recall that~$\E(u; \cdot)$ is monotone non-decreasing with respect to set inclusion.

Now, we deal with the limit on the right-hand side of~\eqref{uuj}.

Let~$\{ \varepsilon_j \}_{j \in \N}$ be the sequence of positive real numbers given by
\begin{equation} \label{epsjdef}
\varepsilon_j := \| u_j - u \|_{L^\infty(B_{R + 1})}.
\end{equation}
Clearly,~$\varepsilon_j$ converges to~$0$ and we may also assume~$\varepsilon_j \le 1 / 2$ for any~$j$. Take~$\eta_j \in C^\infty_c(\R^n)$ to be a cut-off function satisfying~$0 \le \eta_j \le 1$ in~$\R^n$,~$\eta_j = 1$ in~$B_R$,~$\supp(\eta_j) \subseteq B_{R + \varepsilon_j}$ and~$|\nabla \eta_j| \le 2 / \varepsilon_j$ in~$\R^n$. Let~$\varphi$ be as in~\eqref{uclassA} and suppose without loss of generality that~$\varphi \in L^\infty(\R^n)$. We are also allowed to assume~$\E(u + \varphi; B_R) < +\infty$, formula~\eqref{uclassA} being trivially satisfied otherwise. As a consequence of this,~\eqref{Eu},~\eqref{Kbounds} and the boundedness of~$u$ and~$\varphi$, we have that~$\varphi \in H^s(B_{R + 1})$. We define~$v := u + \varphi$ and
$$
v_j := \eta_j u + (1 - \eta_j) u_j + \varphi \quad \mbox{in } \R^n.
$$
Notice that~$v_j = v$ in~$B_R$ and~$v_j = u_j$ in~$\R^n \setminus B_{R + \varepsilon_j}$. Accordingly,~$v_j$ is an admissible competitor for~$u_j$ in~$B_{R + \varepsilon_j}$ and thus
\begin{equation} \label{ujvj}
\E(u_j; B_{R + \varepsilon_j}) \le \E(v_j; B_{R + \varepsilon_j}),
\end{equation}
in view of the minimizing property of~$u_j$. Furthermore,~$v_j$ converges to~$v$ uniformly on compact subsets of~$\R^n$ and, in particular,
$$
\| v_j - v \|_{L^\infty(B_{R + 1})} \le \| u_j - u \|_{L^\infty(B_{R + 1})} = \varepsilon_j.
$$

Fix a number~$\delta \in (0, 1)$ and take~$j$ big enough to have~$\varepsilon_j < \delta / 2$. We address the right-hand side of~\eqref{ujvj}. Concerning its kinetic part, we decompose the domain of integration~$\C_{B_{R + \varepsilon_j}}$ as
\begin{equation} \label{DEFj}
\C_{B_{R + \varepsilon_j}} = D_\delta \cup E_{j, \delta} \cup F_{j, \delta},
\end{equation}
where, up to sets of measure zero,
\begin{align*}
D_\delta & := \left( B_R \times B_R \right) \cup \left( B_R \times \left( B_{R + \delta} \setminus B_R \right) \right) \cup \left( \left( B_{R + \delta} \setminus B_R \right) \times B_R \right), \\
E_{j, \delta} & := \left( \C_{B_{R + \varepsilon_j}} \cap \left( B_{R + \delta} \times B_{R + \delta} \right) \right) \setminus D_\delta, \\
F_{j, \delta} & := \C_{B_{R + \varepsilon_j}} \setminus \left( B_{R + \delta} \times B_{R + \delta} \right).
\end{align*}
See Figure~\ref{DEFdecomp}. Also set
$$
F_\delta := \C_{B_R} \setminus \left( B_{R + \delta} \times B_{R + \delta} \right),
$$
and observe that, analogously to~\eqref{DEFj}, it holds
\begin{equation} \label{DF}
\C_{B_R} = D_\delta \cup F_\delta.
\end{equation}

\begin{figure}[h]
\centering
\fbox{\includegraphics[width=0.9\textwidth]{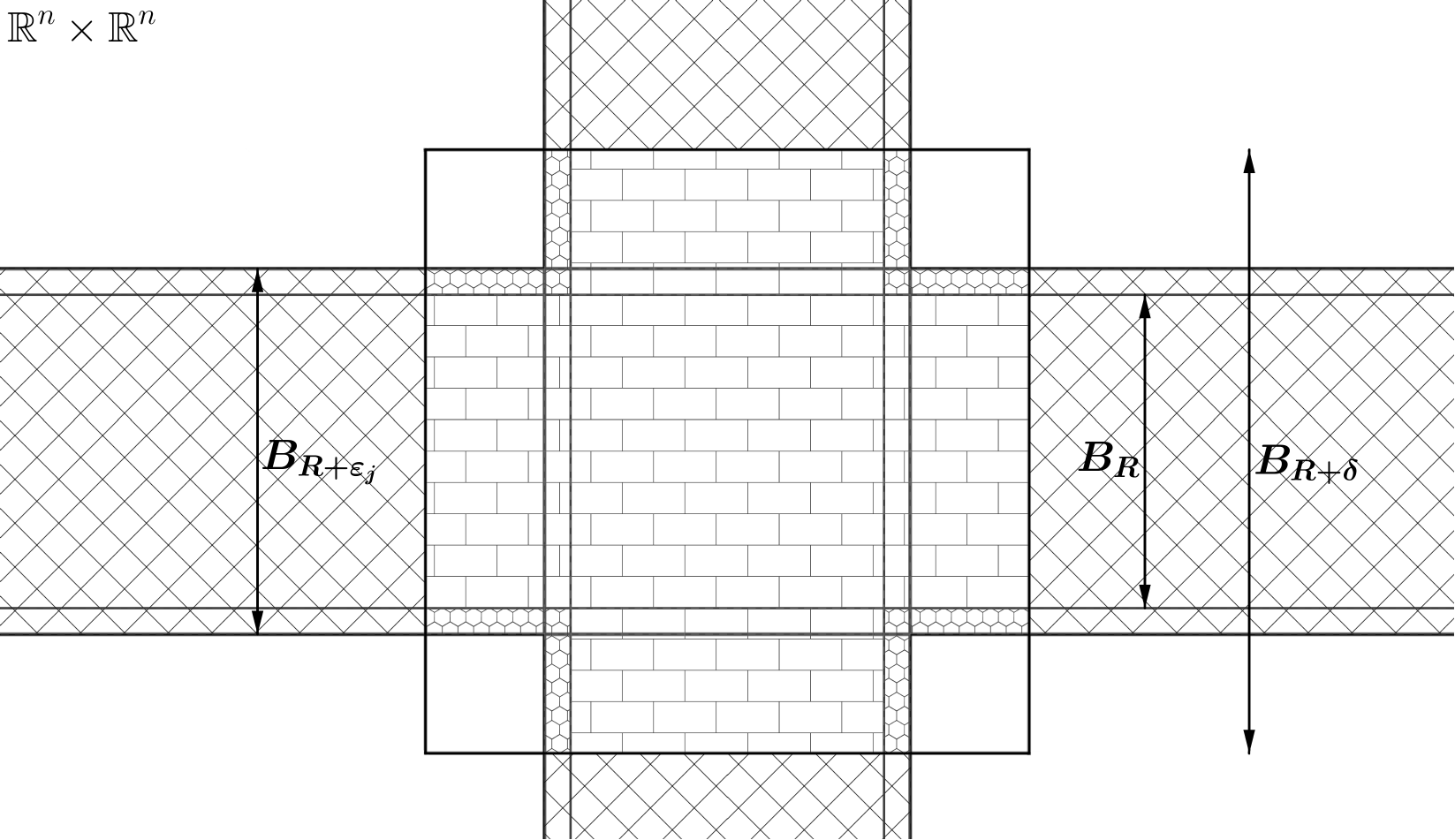}}
\caption{The decomposition of the region~$\C_{B_{R + \varepsilon_j}}$ as given by~\eqref{DEFj}. The set~$D_\delta$ is rendered in the `brick' texture,~$E_{j, \delta}$ in the `honeycomb' one and the `diagonal crosshatch' is used to denote~$F_{j, \delta}$.}
\label{DEFdecomp}
\end{figure}

First, we deal with the tail term of~$\E$, which corresponds to~$F_{j, \delta}$. Note that~$F_{j, \delta}$ may be written as the union of~$B_{R + \varepsilon_j} \times \left( \R^n \setminus B_{R + \delta} \right)$ and~$\left( \R^n \setminus B_{R + \delta} \right) \times B_{R + \varepsilon_j}$. By~\eqref{Ksimm}, it is clearly enough to study what happens inside the first set of this union. Given~$x \in B_{R + \varepsilon_j}$ and~$y \in \R^n \setminus B_{R + \delta}$, we have
$$
|v_j(x) - v_j(y)| = |v_j(x) - u_j(y)| \le 3 + |\varphi(x)|.
$$
Moreover,~$|x| \le R + \varepsilon_j \le [ (R + \delta / 2) / (R + \delta) ] |y|$ and thus
$$
|x - y| \ge |y| - |x| \ge \frac{\delta}{2(R + \delta)} |y|.
$$
Using~\eqref{Kbounds}, for any~$x \in B_{R + 1}$ and~$y \in \R^n \setminus B_{R + \delta}$ we get
\begin{align*}
\left| v_j(x) - v_j(y) \right|^2 K(x, y) \chi_{B_{R + \varepsilon_j}}(x) \le C \, \frac{1 + |\varphi(x)|^2}{|y|^{n + 2 s}} \in L^1 \left( B_{R + 1} \times (\R^n \setminus B_{R + \delta}) \right),
\end{align*}
for some constant~$C > 0$ independent of~$j$. Recalling that~$v_j$ converges pointwise to~$v$ in~$\R^n$, by the Dominated Convergence Theorem we conclude that
\begin{equation} \label{Fvjv}
\lim_{j \rightarrow +\infty} \iint_{F_{j, \delta}} \left| v_j(x) - v_j(y) \right|^2 K(x, y) \, dx dy = \iint_{F_\delta} \left| v(x) - v(y) \right|^2 K(x, y) \, dx dy.
\end{equation}

Now, we focus on~$E_{j, \delta}$. By the triangle inequality, for any~$x, y \in B_{R + 1}$ we write
\begin{align*}
\left| v_j(x) - v_j(y) \right| & \le \left| \eta_j(x) - \eta_j(y) \right| \left| u(x) - u_j(x) \right| + |\eta_j(y)| \left| u(x) - u(y) \right| \\
& \quad + \left| 1 - \eta_j(y) \right| \left| u_j(x) - u_j(y) \right| + \left| \varphi(x) - \varphi(y) \right| \\
& \le \varepsilon_j \left| \eta_j(x) - \eta_j(y) \right| + \left| u(x) - u(y) \right| + \left| u_j(x) - u_j(y) \right| + \left| \varphi(x) - \varphi(y) \right|,
\end{align*}
where we also used~\eqref{epsjdef} and that~$|\eta_j| \le 1$. Hence, taking advantage of~\eqref{Kbounds} and the regularity of~$\eta_j$,
\begin{align*}
& \left[ \iint_{E_{j, \delta}} \left| v_j(x) - v_j(y) \right|^2 K(x, y) \, dx dy \right]^{\frac{1}{2}} \\
& \hspace{20pt} \le \left[ 4 \Lambda \iint_{E_{j, \delta}} \frac{dx dy}{|x - y|^{n - 2 + 2 s}} \right]^{\frac{1}{2}} + \left[ \iint_{E_{j, \delta}} \left| u(x) - u(y) \right|^2 K(x, y) \, dx dy \right]^{\frac{1}{2}} \\
& \hspace{20pt} \quad + \left[ \iint_{E_{j, \delta}} \left| u_j(x) - u_j(y) \right|^2 K(x, y) \, dx dy \right]^{\frac{1}{2}} + \left[ \iint_{E_{j, \delta}} \left| \varphi(x) - \varphi(y) \right|^2 K(x, y) \, dx dy \right]^{\frac{1}{2}}.
\end{align*}
Note that the arguments of the first, second and fourth integrals on the right-hand side above are integrable on the set~$B_{R + 1} \times B_{R + 1}$, which contains~$E_{j, \delta}$. Thus, by the absolute continuity of the Lebesgue measure in~$\R^n \times \R^n$, it follows that those integrals go to zero, as~$j \rightarrow +\infty$ (observe in this regard that~$|E_{j, \delta}| \rightarrow 0$). Moreover, in view of~\eqref{Eu_j}, we conclude that
\begin{equation} \label{Evjv}
\iint_{E_{j, \delta}} \left| v_j(x) - v_j(y) \right|^2 K(x, y) \, dx dy \le \iint_{E_{j, \delta}} \left| u_j(x) - u_j(y) \right|^2 K(x, y) \, dx dy + 2 \rho_j,
\end{equation}
for some sequence~$\{ \rho_j \}$ of positive real numbers such that
\begin{equation} \label{rholimit}
\lim_{j \rightarrow +\infty} \rho_j = 0.
\end{equation}

We are left with the term involving~$D_\delta$. We recall that~$v_j = v$ in~$B_R$, so that
\begin{equation} \label{BRvjv}
\int_{B_R} \int_{B_R} \left| v_j(x) - v_j(y) \right|^2 K(x, y) \, dx dy = \int_{B_R} \int_{B_R} \left| v(x) - v(y) \right|^2 K(x, y) \, dx dy.
\end{equation}
Therefore, we just need to examine the complement~$D_\delta \setminus \left( B_R \times B_R \right)$ and thus, by symmetry, the region~$B_R \times \left( B_{R + \delta} \setminus B_R \right)$ only. Letting~$x \in B_R$ and~$y \in B_{R + \delta} \setminus B_R$, by~\eqref{epsjdef} we have
\begin{align*}
\left| v_j(x) - v_j(y) \right| & = \left| v(x) - v_j(y) \right| \le \left| v(x) - v(y) \right| + \left| 1 - \eta_j(y) \right| \left| u(y) - u_j(y) \right| \\
& = \left| v(x) - v(y) \right| + \left| \eta_j(x) - \eta_j(y) \right| \left| u(y) - u_j(y) \right| \\
& \le \left| v(x) - v(y) \right| + \varepsilon_j \left| \eta_j(x) - \eta_j(y) \right|.
\end{align*}
Then, by the definition of~$\eta_j$ and~$\eqref{Kbounds}$ we get
\begin{align*}
& \left[ \int_{B_R} \int_{B_{R + \delta} \setminus B_R} \left| v_j(x) - v_j(y) \right|^2 K(x, y) \, dx dy \right]^{\frac{1}{2}} \\
& \hspace{30pt} \le \left[ \int_{B_R} \int_{B_{R + \delta} \setminus B_R} \left| v(x) - v(y) \right|^2 K(x, y) \, dx dy \right]^{\frac{1}{2}} + \left[ \int_{B_R} \int_{B_{R + \delta} \setminus B_R} \frac{4 \Lambda \, dx dy}{|x - y|^{n - 2 + 2 s}} \right]^{\frac{1}{2}} \\
& \hspace{30pt} \le \left[ \int_{B_R} \int_{B_{R + \delta} \setminus B_R} \left| v(x) - v(y) \right|^2 K(x, y) \, dx dy \right]^{\frac{1}{2}} + C \left| B_{R + \delta} \setminus B_R \right|^{\frac{1}{2}},
\end{align*}
for some constant~$C > 0$ independent of~$j$ and~$\delta$. Recalling~\eqref{BRvjv}, we may thence conclude that there exists a function~$r: (0, 1) \to (0, +\infty)$ for which
\begin{equation} \label{rlimit}
\lim_{\delta \rightarrow 0^+} r(\delta) = 0,
\end{equation}
and
\begin{equation} \label{Dvjv}
\iint_{D_\delta} \left| v_j(x) - v_j(y) \right|^2 K(x, y) \, dx dy \le \iint_{D_\delta} \left| v(x) - v(y) \right|^2 K(x, y) \, dx dy + 2 r(\delta),
\end{equation}
for any~$j$ big enough.

Observe now that for the potential term of~$\E$ we may simply estimate
$$
\P(v_j; B_{R + \varepsilon_j}) \le \P(v; B_R) + W^* \left| B_{R + \varepsilon_j} \setminus B_R \right|.
$$
Taking advantage of decomposition~\eqref{DEFj} on both sides of~\eqref{ujvj} and using inequalities~\eqref{Evjv},~\eqref{Dvjv}, we write
\begin{align*}
& \frac{1}{2} \iint_{D_\delta \cup E_{j, \delta} \cup F_{j, \delta}} \left| u_j(x) - u_j(y) \right|^2 K(x, y) \, dx dy + \P(u_j; B_{R + \varepsilon_j}) \\
& \hspace{15pt} = \E(u_j; B_{R + \varepsilon_j}) \le \E(v_j; B_{R + \varepsilon_j}) \\
& \hspace{15pt} \le \frac{1}{2} \iint_{D_\delta} \left| v(x) - v(y) \right|^2 K(x, y) \, dx dy + \frac{1}{2} \iint_{E_{j, \delta}} \left| u_j(x) - u_j(y) \right|^2 K(x, y) \, dx dy \\
& \hspace{15pt} \quad + \frac{1}{2} \iint_{F_{j, \delta}} \left| v_j(x) - v_j(y) \right|^2 K(x, y) \, dx dy + \P(v; B_R) + W^* \left| B_{R + \varepsilon_j} \setminus B_R \right| + r(\delta) + \rho_j,
\end{align*}
which in turn simplifies to
\begin{align*}
& \frac{1}{2} \iint_{D_\delta \cup F_{j, \delta}} \left| u_j(x) - u_j(y) \right|^2 K(x, y) \, dx dy + \P(u_j; B_{R + \varepsilon_j}) \\
& \hspace{30pt} \le \frac{1}{2} \iint_{D_\delta} \left| v(x) - v(y) \right|^2 K(x, y) \, dx dy + \frac{1}{2} \iint_{F_{j, \delta}} \left| v_j(x) - v_j(y) \right|^2 K(x, y) \, dx dy \\
& \hspace{30pt} \quad + \P(v; B_R) + W^* \left| B_{R + \varepsilon_j} \setminus B_R \right| + r(\delta) + \rho_j.
\end{align*}
If we exploit the fact that~$\C_{B_R} \subset D_\delta \cup F_{j, \delta}$ and recall~\eqref{DF},~\eqref{Fvjv},~\eqref{rholimit}, by taking the limit in~$j$ in the previous formula we find
$$
\limsup_{j \rightarrow +\infty} \E(u_j; B_R) \le \E(v; B_R) + r(\delta).
$$
Putting together this last inequality with~\eqref{uuj}, we finally obtain
$$
\E(u; B_R) \le \E(v; B_R) + r(\delta).
$$
Then,~\eqref{uclassA} follows from the arbitrariness of~$\delta$ and~\eqref{rlimit}. We conclude that~$u$ is a class~A minimizer of~$\E$.

\counterwithin{definition}{section}
\counterwithin{equation}{section}

\section{Proof of Theorem~\ref{mainthm} for general kernels} \label{s<1/2sec}

In this section we complete the proof of Theorem~\ref{mainthm}, by extending the results of Section~\ref{s>1/2sec} to kernels which do not necessarily satisfy condition~\eqref{Krapid}. This can be done in consequence of the fact that none of the estimates established there involve any of the parameters appearing in~\eqref{Krapid}. This enables us to perform a limit argument analogous to that of Subsection~\ref{irrsubsec}.

\medskip

Let~$K$ be a kernel satisfying~\eqref{Ksimm},~\eqref{Kbounds} and~\eqref{Kper} only. Given any monotone increasing sequence~$\{ R_j \}_{j \in \N} \subset [2, +\infty)$ which diverges to~$+\infty$, we set
$$
K_j(x, y) := K(x, y) \chi_{[0, R_j]}(|x - y|) \quad \mbox{for any } x, y \in \R^n.
$$
Notice that the new truncated kernel~$K_j$ still satisfies hypotheses~\eqref{Ksimm},~\eqref{Kbounds} and~\eqref{Kper}. Moreover,~$K_j$ clearly fulfills the additional requirement~\eqref{Krapid} with~$\bar{R} = R_j$.

Let~$\E_j$ be the energy functional~\eqref{EOmegadef} corresponding to~$K_j$. For a fixed direction~$\omega \in \R^n \setminus \{ 0 \}$, let~$u_j$ be the plane-like class A minimizer for~$\E_j$ directed along~$\omega$. The existence of such minimizers is a consequence of Section~\ref{s>1/2sec}, as~$K_j$ verifies~\eqref{Krapid}. It holds
\begin{equation} \label{ujplanelikebis}
\left\{ x \in \R^n : \left| u_j(x) \right| \le \frac{9}{10} \right\} \subseteq \left\{ x \in \R^n : \frac{\omega}{|\omega|} \cdot x \in [0, M_0] \right\},
\end{equation}
for a universal value~$M_0 > 0$. Furthermore,~$|u_j| \le 1$ in~$\R^n$ and, in view of Corollary~\ref{minminholdercor},~$\| u_j \|_{C^{0, \alpha}(\R^n)} \le C$, for some~$\alpha \in (0, 1]$ and~$C > 0$. We highlight the fact that we can choose~$M_0$,~$\alpha$ and~$C$ to be independent of~$j$, since each~$K_j$ satisfies~\eqref{Kbounds} with the same structural constants. Accordingly, by Arzel\`a-Ascoli Theorem~$\{ u_j \}$ converges, up to a subsequence, to a continuous function~$u$, uniformly on compact subset of~$\R^n$.

Observe that~$u$ satisfies~\eqref{ujplanelikebis}. Also, if~$\omega$ is rational then each~$u_j$ is~$\sim$-periodic and, consequently, so is~$u$.
To prove that~$u$ is a class~A minimizer, fix~$R \ge 1$ and consider a perturbation~$\varphi$, with~$\supp(\varphi) \subset \subset B_R$. We know that
$$
\E_j(u_j; B_R) \le \E_j(u_j + \varphi; B_R) \quad \mbox{for any } j \in \N.
$$
On the one hand, a simple application of Fatou's lemma implies that
$$
\E(u; B_R) \le \liminf_{j \rightarrow +\infty} \E_j(u_j; B_R).
$$
On the other hand, following the strategy presented in Subsection~\ref{irrsubsec} it is not hard to see that we also have
$$
\limsup_{j \rightarrow +\infty} \E_j(u_j; B_R) \le \E(u + \varphi; B_R).
$$
It follows that~$u$ is a class~A minimizer of~$\E$ and the proof of Theorem~\ref{mainthm} is therefore complete.

\section{Stability of Theorem~\ref{mainthm} as~$s$ approaches~$1$} \label{sto1sec}

In this brief section we discuss what happens when we take the limit as~$s \rightarrow 1^-$ in Theorem~\ref{mainthm}. Since (at least for some choices of~$K$) the energy in~\eqref{Edef} becomes closer and closer to a local gradient functional, as~$s$ approaches~$1$, one expects to recover the result of~\cite{V04} in the limit. While this is certainly true, the rigorous computation supporting this intuition is not completely trivial. We include it here in for the reader's convenience.

\medskip

We restrict ourselves to consider the simpler case determined by the family of kernels
$$
K_s(x, y) := \frac{1 - s}{|x - y|^{n + 2 s}}.
$$
Corresponding to these choices, we have the energy functionals
\begin{equation} \label{Esdef}
\E_s(u; \Omega) := \frac{1}{2} \iint_{\C_\Omega} |u(x) - u(y)|^2 K_s(x, y) \, dx dy + \int_\Omega W(x, u(x)) \, dx,
\end{equation}
defined for any measurable set~$\Omega \subset \R^n$ and with~$\C_\Omega$ as in~\eqref{COmega}.

As~$s \rightarrow 1^-$, we expect (see e.g.~\cite{BBM01}) the energy~$\E_s$ to converge in some sense to the local functional
\begin{equation} \label{limitfunc}
\E(u; \Omega) := \frac{C_\star}{2} \int_\Omega |\nabla u(x)|^2 \, dx + \int_\Omega W(x, u(x)) \, dx,
\end{equation}
for some dimensional constant\footnote{To be precise,~$C_\star$ is the constant denoted with~$K$ in~\cite[Corollary~2]{BBM01} and with~$K_{2, N}$ in~\cite[Formula~(3)]{P04}, up to a multiplicative dimensional constant. Its value is~$C_\star := \frac{1}{2} \int_{\partial B_1} |e_1 \cdot \sigma|^2 \, d\Haus^{n - 1}(\sigma)$.}~$C_\star > 0$. Notice in particular the factor~$1 - s$ appearing in the definition of~$K_s$, that corrects the energy and prevents its blow-up, as~$s \rightarrow 1$.

In the following, we show how~\cite[Theorem~8.1]{V04} for the energy defined in~\eqref{limitfunc} can be recovered from Theorem~\ref{mainthm} here, applied to the family of functionals~$\E_s$ of~\eqref{Esdef}.

Note that in~\cite[Theorem~8.1]{V04} the author proves the existence of plane-like minimizers for a far more general class of Ginzburg-Landau-type functionals than those comprised by~\eqref{limitfunc}, by allowing for instance the presence of a non-homogeneous gradient term such as~\eqref{localkin}. Although we believe it would be very interesting to investigate how such larger class of local functionals can be approximated by non-local ones, this goes well beyond the scopes of the present section, in which we aim to give just a glimpse of how our result compares with that of~\cite{V04}. However, we stress that the generality covered by~\eqref{Esdef} and~\eqref{limitfunc} still is rather wide and meaningful in relation to plane-like minimizers which are not one-dimensional, due to the presence of the space-dependent potential~$W$.

\smallskip

We are now ready to state and prove the following result.

\begin{theorem} \label{sto1thm}
Let~$n \ge 2$ and assume that~$W$ satisfies~\eqref{Wzeros},~\eqref{Wgamma},~\eqref{Wbound},~\eqref{Wper}. Fix any value~$\theta \in (0, 1)$ and any direction~$\omega \in \R^n \setminus \{ 0 \}$. For any~$s \in (0, 1)$, let~$u_s$ be the plane-like class~A minimizer of the energy~$\E_s$, associated with~$\theta$ and~$\omega$, as given by Theorem~\ref{mainthm}. Then, there exists an increasing sequence~$\{ s_k \}_{k \in \N}$ converging to~$1$, such that~$u_{s_k}$ converges in~$C^1_\loc(\R^n)$ to some function~$u: \R^n \to [-1, 1]$, as~$k \rightarrow +\infty$.\\ Furthermore,~$u$ is a class~A minimizer\footnote{Of course, the notions of local and class~A minimizer of the functional~$\E$ defined in~\eqref{limitfunc} are very classical and indeed quite similar to those introduced in Definitions~\ref{mindef} and~\ref{classAdef} for non-local energies. For us, a class~A minimizer of~$\E$ is a function~$u$ for which~$\E(u; \Omega) < +\infty$ and~$\E(u; \Omega) \le \E(v; \Omega)$ for any~$v$ that coincides with~$u$ outside of~$\Omega$, for any bounded set~$\Omega \subset \R^n$.} of~$\E$ satisfying
\begin{equation} \label{limitupl}
\bigg\{ x \in \R^n : |u(x)| < \theta \bigg\} \subset \bigg\{ x \in \R^n : \frac{\omega}{|\omega|} \cdot x \in [0, M_0] \bigg\},
\end{equation}
for some constant~$M_0 > 0$ that depends only on~$n$,~$W^*$, the function~$\gamma$ and~$\theta$.
\end{theorem}

Theorem~\ref{sto1thm} yields the convergence of the plane-like minimizers of~$\E_s$ to those of~$\E$ and establishes, as a byproduct of Theorem~\ref{mainthm}, the existence of the latter. In this way, the main result of~\cite{V04} holds as a consequence of Theorem~\ref{mainthm}.

Before heading to the proof of Theorem~\ref{sto1thm}, we first address the validity of the following auxiliary result.

\begin{lemma} \label{intestlem}
Let~$\Omega \subset \subset \Omega'$ be bounded open subsets of~$\R^n$, with~$\Omega$ having Lipschitz boundary. Let~$\{ s_k \}_{k \in \N} \subset (1/4, 1)$ be a sequence converging to~$1$ and~$\{ w_k \}_{k \in \N}$ be a sequence of functions, bounded in~$L^\infty(\R^n) \cap C^{0, 1}(\Omega')$. Then,
$$
\lim_{k \rightarrow +\infty} \int_{\Omega} \int_{\R^n \setminus \Omega} |w_k(x) - w_k(y)|^2 K_{s_k}(x, y) \, dx dy = 0.
$$
\end{lemma}
\begin{proof}
Let~$\varepsilon > 0$ be a small number to be chosen later. In what follows, we indicate with~$c$ any positive constant that does not depend on neither~$k$ nor~$\varepsilon$. By our assumptions on~$\{ w_k \}$, we have that
$$
|w_k(x) - w_k(y)| \le c \Big[ |x - y| \chi_{[0, \varepsilon)}(|x - y|) + \chi_{[\varepsilon, +\infty)}(|x - y|) \Big] \quad \mbox{for any } x \in \Omega, \, y \in \R^n,
$$
provided~$\varepsilon$ is sufficiently small. By this, we compute
\begin{equation} \label{intleI+J}
\int_{\Omega} \int_{\R^n \setminus \Omega} |w_k(x) - w_k(y)|^2 K_{s_k}(x, y) \, dx dy \le c (1 - s_k) \left( I_{\varepsilon} + J_\varepsilon \right),
\end{equation}
where
\begin{align*}
I_\varepsilon & := \int_{\Omega} \left( \int_{B_\varepsilon(x) \setminus \Omega} \frac{dy}{|x - y|^{n - 2 + 2 s_k}} \right) dx,\\
\mbox{and } \, J_\varepsilon & := \int_{\Omega} \left( \int_{\R^n \setminus B_\varepsilon(x)} \frac{dy}{|x - y|^{n + 2 s_k}} \right) dx.
\end{align*}
By noticing that
$$
x \in \Omega, \, y \in B_\varepsilon(x) \setminus \Omega \quad \mbox{implies that} \quad x \in \Omega_\varepsilon := \Big\{ x \in \Omega : \dist(x, \partial \Omega) < \varepsilon \Big\},
$$
and changing variables appropriately, we estimate the first integral as follows:
$$
I_\varepsilon \le \int_{\Omega_\varepsilon} \left( \int_{B_\varepsilon} \frac{dz}{|z|^{n - 2 + 2 s_k}} \right) dx \le c \, \frac{\varepsilon^{3 - 2 s_k}}{1 - s_k}.
$$
Note that we took advantage of the Lipschitzianity of~$\partial \Omega$ to deduce the last inequality. In a similar (and easier) way, we also obtain
$$
J_\varepsilon \le c \, \varepsilon^{- 2 s_k}.
$$
By combining these last two inequalities with~\eqref{intleI+J}, we get
$$
\int_{\Omega} \int_{\R^n \setminus \Omega} |w_k(x) - w_k(y)|^2 K_{s_k}(x, y) \, dx dy \le c \left( 1 + \frac{1 - s_k}{\varepsilon^3} \right) \varepsilon^{3 - 2 s_k}.
$$
Select now~$\varepsilon = \varepsilon_k := \sqrt[3]{1 - s_k}$. By plugging this in the last formula, we end up with
$$
\int_{\Omega} \int_{\R^n \setminus \Omega} |w_k(x) - w_k(y)|^2 K_{s_k}(x, y) \, dx dy \le c (1 - s_k)^{1 - \frac{2}{3} s_k} \le c \sqrt[3]{1 - s_k},
$$
and the thesis readily follows.
\end{proof}

With the aid of Lemma~\ref{intestlem}, we can now prove the main result of the present section.

\begin{proof}[Proof of Theorem~\ref{sto1thm}]
First, we observe that, by the regularity theory for the 
fractional Laplacian (see e.g.~\cite[Theorem~6.1]{CS11}),
the minimizers~$u_s$ belong to~$C^{1, \alpha}(\R^n)$, with~$\alpha > 0$ independent of~$s$, and actually form a bounded family in that space, for, say,~$s > 3/4$. Note that the result of~\cite{CS11} holds in principle for viscosity solutions. But this notion is indeed equivalent to bounded weak solutions, 
when dealing with bounded, continuous right-hand sides (see~\cite{SerV14}; see also~\cite{DKV14}
for related results based on Moser's iteration). In our case, the~$u_s$'s are bounded weak solutions of equations with right-hand sides given by~$W_r(\cdot, u_s)$, which are bounded and continuous, since the~$u_s$'s are, thanks to Theorem~\ref{holderthm}.

By Arzel\`a-Ascoli Theorem, then there exists a sequence~$\{ s_k \}_{k \in \N}$ increasing to~$1$, such that~$u_{s_k}$ converge in~$C^1_\loc(\R^n)$ to some differentiable function~$u$. Observe that~$u$ satisfies~\eqref{limitupl}. To see this, it is sufficient to notice that the~$u_s$'s satisfy an analogous inclusion, with~$M_0$ independent of~$s$, for~$s$ close to~$1$. But this is true, as one can check by inspecting the proof of Theorem~\ref{mainthm}, when applied to the functionals~$\E_s$ (observe in particular that the constants~$C_1$ and~$C_2$ appearing in~\eqref{holderunipre} and~\eqref{minminenest}, respectively, may be chosen independently of~$s$).

To conclude the proof, we are therefore only left to show that~$u$ is a class~A minimizer of the energy~$\E$ given by~\eqref{limitfunc}. To this aim, let~$R \ge 4$ and~$v$ be a function coinciding with~$u$ outside of the ball~$B_R$. We need to prove that
\begin{equation} \label{limitulev}
\E(u; B_R) \le \E(v; B_R).
\end{equation}
By standard density results in Sobolev spaces, we may suppose without loss of generality that~$v \in C^1(\overline{B_R})$. In view of the regularity of~$u$ and~$v$, we also have that~$v \in C^{0, 1}(\R^n)$

To deduce~\eqref{limitulev}, we modify~$v$ outside a ball containing~$B_R$ in order to obtain a sequence of functions coinciding with the~$u_{s_k}$'s there and be in position to take advantage of the minimizing properties of each~$u_{s_k}$. The technical details of this construction are presented here below.

For a fixed~$\delta \in (0, 1)$, we consider a radially symmetric and non-increasing cut-off function~$\eta = \eta_\delta \in C^\infty(\R^n)$, with~$\supp(\eta) \subset B_{R + \delta}$,~$\eta = 1$ in~$B_R$ and~$|\nabla \delta| \le 2 / \delta$. We set
$$
v_{s, \delta} := \eta v + (1 - \eta) u_s.
$$
Note that~$v_{s, \delta} \in C^{0, 1}(\R^n)$, with Lipschitz constant bounded uniformly in~$s$ (but not in~$\delta$). A straightforward computation shows that it holds in particular
\begin{equation} \label{vsdeltaLip}
|v_{s, \delta}(x) - v_{s, \delta}(y)| \le c_\star \left( 1 + \frac{|u_s(y) - v(y)|}{\delta} \right) |x - y|,
\end{equation}
for some constant~$c_\star > 0$ independent of both~$s$ and~$\delta$.

Thanks to the results of~\cite[Section~3]{BBM01} or~\cite[Theorem~1.2]{P04}, we have that
$$
\E(u; B_R) \le \E(u; B_{R + \delta}) \le \liminf_{k \rightarrow +\infty} \E_{s_k}(u_{s_k}; B_{R + \delta}).
$$
As~$u_s$ is a class~A minimizer for~$\E_s$ and~$v_{s, \delta}$ coincides with~$u_s$ outside of~$B_{R + \delta}$, we then obtain that
$$
\E(u; B_R) \le \liminf_{k \rightarrow +\infty} \E_{s_k}(v_{s_k, \delta}; B_{R + \delta}).
$$
Finally, as~$v_{s, \delta}$ coincides with~$v$ inside~$B_R$, using~\cite[Corollary~2]{BBM01} we conclude that
\begin{equation} \label{ulev+rest}
\E(u; B_R) \le \E(v; B_R) + \RR_\delta,
\end{equation}
where
\begin{align*}
\RR_\delta & := \limsup_{k \rightarrow +\infty} \RR_{k, \delta},\\ \mbox{with } \, \RR_{k, \delta} & := \int_{B_{R + \delta}} \int_{\R^n \setminus B_R} |v_{s_k, \delta}(x) - v_{s_k, \delta}(y)|^2 K_{s_k}(x, y) \, dx dy.
\end{align*}

To deduce the validity of~\eqref{limitulev} from~\eqref{ulev+rest}, we therefore only need to prove that the remainder term~$\RR_\delta$ goes to zero, as~$\delta \rightarrow 0^+$. To do this, we write~$\RR_{k, \delta} = \RR_{k, \delta}^{(1)} + \RR_{k, \delta}^{(2)}$, where
\begin{align*}
\RR_{k, \delta}^{(1)} & := \int_{B_{R + \delta}} \int_{\R^n \setminus B_{R + \delta}} |v_{s_k, \delta}(x) - v_{s_k, \delta}(y)|^2 K_{s_k}(x, y) \, dx dy,\\
\mbox{and } \, \RR_{k, \delta}^{(2)} & := \int_{B_{R + \delta}} \int_{B_{R + \delta} \setminus B_R} |v_{s_k, \delta}(x) - v_{s_k, \delta}(y)|^2 K_{s_k}(x, y) \, dx dy.
\end{align*}
In view of Lemma~\ref{intestlem} and the fact that the~$v_{s, \delta}$'s are bounded in~$C^{0, 1}(\R^n)$ uniformly in~$s$, we know that
$$
\lim_{k \rightarrow +\infty} \RR_{k, \delta}^{(1)} = 0,
$$
for any~$\delta > 0$. Hence, to conclude that~$\RR_\delta \rightarrow 0$, we just need to inspect the contributions coming from~$\RR_{k, \delta}^{(2)}$. Indeed, we claim that, for any~$\delta \in (0, 1/2)$,
\begin{equation} \label{R2vanishes}
\limsup_{k \rightarrow +\infty} \RR_{k, \delta}^{(2)} \le C \delta,
\end{equation}
for some constant~$C > 0$ independent of~$\delta$. Note that if we establish this, then~\eqref{limitulev} would follow.

In order to check~\eqref{R2vanishes}, we take~$k_\delta \in \N$ sufficiently large to have
$$
\| u_{s_k} - u \|_{L^\infty(B_{R + 1})} \le \delta \quad \mbox{for any } k \ge k_\delta.
$$
By this,~\eqref{vsdeltaLip} and the fact that~$v = u$ outside of~$B_R$, for~$k \ge k_\delta$ we compute
\begin{align*} \label{limittech1}
\RR_{k, \delta}^{(2)} & \le c_\star^2 \int_{B_{R + \delta}} \left( \int_{B_{R + \delta} \setminus B_R} \left( 1 + \frac{|u_{s_k}(y) - u(y)|}{\delta} \right)^2 |x - y|^2 K_{s_k}(x, y) \, dy \right) dx \\
& \le 4 c_\star^2 (1 - s_k) \int_{B_{R + \delta} \setminus B_R} \left( \int_{B_{R + 1}} \frac{dx}{|x - y|^{n - 2 + 2 s_k}} \right) dy \\
& \le 4 c_\star^2 |B_{R + \delta} \setminus B_R| (1 - s_k) \int_{B_{2 R + 2}} \frac{dz}{|z|^{n - 2 + 2 s_k}} \\
& \le C \delta,
\end{align*}
for some~$C > 0$ independent of~$k$ and~$\delta$. This clearly implies~\eqref{R2vanishes} and the proof of Theorem~\ref{sto1thm} is thence complete.
\end{proof}

\section{Note added in proof. Weakening of some structural assumptions} \label{AiPsec}

As a consequence of the results obtained in~\cite{C17b} by the first author,\footnote{We emphasize that~\cite{C17b} was not yet available at the time a previous, but finished, version of the present manuscript was completed. We preferred to add this section, instead of altering the core parts of the paper in the light of~\cite{C17b}, mainly to preserve the correct chronological timeline.} some of the hypotheses listed in the introduction can be slightly relaxed. Indeed, the differentiability of the potential~$W$ is no longer needed for the proof of the main result of this paper.

More specifically, Theorem~\ref{mainthm} continues to hold if we replace assumption~\eqref{Wbound} with the following weaker requirement:
\begin{equation} \tag{W3$\ensuremath{'}$} \label{W3'}
\begin{gathered}
\mbox{the map } [-1, 1] \ni r \longmapsto W(x, r) \mbox{ is continuous for a.a. } x \in \R^n, \\
\mbox{and } W(x, r) \le W^* \mbox{ for a.a. } x \in \R^n \mbox{ and any } r \in [-1, 1],
\end{gathered}
\end{equation}
for some~$W^* > 0$.

The reason for this is that the differentiability of~$W$ with respect to the~$r$ variable and the uniform bound for its derivative provided by~\eqref{Wbound} are only used in the proof of Theorem~\ref{mainthm} to apply the regularity theory contained in Section~\ref{regsec}. Since we can now deduce the H\"older continuity of the minimizers of the functional~$\E$ defined in~\eqref{Edef} by taking advantage of~\cite[Theorem~2.4]{C17b} - and therefore not using the Euler-Lagrange equation associated to~$\E$ -, the boundedness of the potential~$W$ granted by~\eqref{W3'} is sufficient.

In consequence of this improvement, the whole range of exponents~$d > 0$ is now admissible in the first example of~\eqref{Wexample}.

\smallskip

We stress that other generalizations of the model considered here could be addressed.

For instance, one can take into account potentials~$W$ that are bounded, but not even continuous, such as~$W(x, r) = Q(x) \chi_{(-1, 1)}(r)$, with~$Q$ positive and periodic. In the local setting, energies involving this potential term are used in the modeling of jets of fluid, and have been studied for instance in~\cite{V04,PV05}. Note that the regularity theory for nonlocal functionals with discontinuous potentials is already available, thanks to~\cite{C17b}. However, Theorem~\ref{mainthm} cannot be automatically extended to these functionals, as the proof provided here makes use of the continuity of~$W$. Nevertheless, we do believe that, with appropriate modifications in the argument, this difficulty could be circumvented.

Another interesting line of investigation is represented by the possibility of replacing the Gagliardo-type seminorm in~\eqref{Edef} with a more general non-quadratic interaction term. The existence of plane-like minimizers for energies with~$L^p$-type gradient structure has been proved in~\cite{PV05}. For nonlocal functionals, we plan to address this problem in a future work.
\smallskip

Moreover, under suitable additional assumptions, in the forthcoming
paper~\cite{CV17}
we will improve the quantitative results of this paper
by showing that
the oscillations of the interfaces with respect to
the reference hyperplane are not only bounded, but
bounded explicitly
by a universal constant times the periodicity scale of the medium.
This additional and quantitative
geometric property will allow us to establish, 
in the limit, the existence of planelike nonlocal minimal 
surfaces in a periodic structure.

\appendix

\section{Some auxiliary results} \label{addcompapp}

In this first appendix we enclose a couple of lemmata which cover some technical aspects that we faced throughout the paper.

\medskip

We begin with an observation on the necessity of hypothesis~\eqref{Krapid} for the validity of the computations of Section~\ref{s>1/2sec}. We refer to Subsection~\ref{minpersubsec}, in particular, for the notation employed in the statement.

\begin{lemma} \label{rapidneclem}
Assume that~$K$ is a measurable kernel satisfying
$$
K(x, y) \ge \frac{\gamma}{|x - y|^{n + \beta}} \quad \mbox{for a.e. } x, y \in \R^n \mbox{ such that } |x - y| \ge \bar{R}, \mbox{ with } \beta \in (0, 1],
$$
for some~$\gamma, \bar{R} > 0$. Then, given any two real numbers~$A < B$, it holds
\begin{equation} \label{rapidnec}
\int_{\{ \omega \cdot x \le A \}} \int_{\tRn \cap \{ \omega \cdot x \ge B \}} \left| u(x) - u(y) \right|^2 K(x, y) \, dx dy = +\infty,
\end{equation}
for any~$u \in \AAB$. Consequently,~$\F \equiv +\infty$ on~$\AAB$.
\end{lemma}
\begin{proof}
Of course, we may take~$\omega = e_n$,~$A = 0$ and~$B = 1$. Then,
$$
\{ \omega \cdot x \le A \} = \R^{n - 1} \times (-\infty, 0] \quad \mbox{and} \quad \tRn \cap \{ \omega \cdot x \ge B \} = [0, 1]^{n - 1} \times [1, +\infty).
$$
Under these conditions, the left~hand side of~\eqref{rapidnec} is controlled from below by
$$
I := \gamma \int_{[0, 1]^{n - 1} \times [1, +\infty)} \left( \int_{\left( \R^{n - 1} \times (-\infty, 0] \right) \setminus B_{\bar{R}}(x)} \frac{\left| u(x) - u(y) \right|^2}{|x - y|^{n + \beta}} \, dy \right) dx.
$$
Since~$u \in \mathcal{A}_{e_n}^{0, 1}$, it follows that for any~$x, y \in \R^n$ such that~$x_n \ge 1$ and~$y_n \le 0$,
$$
\left| u(x) - u(y) \right| = u(y) - u(x) \ge \frac{9}{10} - \left( - \frac{9}{10} \right) = \frac{9}{5} \ge 1.
$$
Hence,
$$
I \ge \gamma \int_{[0, 1]^{n - 1} \times [\bar{R} + 1, +\infty)} \left( \int_{\R^{n - 1} \times (-\infty, 0]} \frac{dy}{|x - y|^{n + \beta}} \right) dx.
$$
Arguing as in the proof of Lemma~\eqref{finfunlem}, it is easy to check that
$$
\int_{\R^{n - 1} \times (-\infty, 0]} \frac{dy}{|x - y|^{n + \beta}} = c x_n^{- \beta},
$$
for some constant~$c > 0$ independent of~$x$. Accordingly,
$$
I \ge c \gamma \int_{\bar{R} + 1}^{+ \infty} x_n^{- \beta} dx_n = +\infty,
$$
since~$\beta \le 1$. The thesis then follows.
\end{proof}

Next is a lemma that ensures the finiteness of the integral appearing on the right-hand side of~\eqref{EFrel}, in Subsection~\ref{cptsubsec}.

\begin{lemma} \label{intfinlem}
Let~$\varphi \in L^\infty(\R^n)$ have support compactly contained in~$\tSm$, in the sense of footnote~\ref{cptfoot} at page~\pageref{cptfoot}. Denote with~$\tilde{\varphi}$ the~$\sim_m$-periodic extension to~$\R^n$ of~$\varphi|_{\tRnm}$. Then, the integral
\begin{equation} \label{finint}
\int_{\tRnm} \int_{\R^n \setminus \tRnm} \frac{|\tilde{\varphi}(x)| |\tilde{\varphi}(y)|}{|x - y|^{n + 2 s}} \, dx dy,
\end{equation}
is finite.
\end{lemma}
\begin{proof}
Assume for simplicity that~$\omega = e_n$ and~$m = \left( 1, \ldots, 1 \right)$. With this choices, we identify~$\tRn$ with its fundamental region~$Q'_{1/2} \times \R$.

We split the domain of integration of~\eqref{finint} as
$$
\tRn \times \left( \R^n \setminus \tRn \right) = \left( \tRn \times \D_1 \right) \cup \left( \tRn \times \D_2 \right),
$$
with
$$
\D_1 := \left( Q'_{\sqrt{n - 1}} \setminus Q'_{1/2} \right) \times \R \quad \mbox{and} \quad \D_2 := \left( \R^{n - 1} \setminus Q'_{\sqrt{n - 1}} \right) \times \R.
$$

We first deal with the integral involving the region~$\D_1$. In view of the hypothesis on the support of~$\varphi$, we have
$$
\dist \left( \overline{\supp (\varphi)}, \D_1 \right) \ge \delta,
$$
for some~$\delta > 0$. Therefore, we estimate
\begin{align*}
\int_{\tRn} \int_{\D_1} \frac{|\tilde{\varphi}(x)| |\tilde{\varphi}(y)|}{|x - y|^{n + 2 s}} \, dx dy & \le \| \varphi \|_{L^\infty(\R^n)}^2 \int_{\supp(\varphi)} \int_{\D_1 \cap \{ x_n \in [A, B] \}} \frac{dx dy}{|x - y|^{n + 2 s}} \\
& \le \| \varphi \|_{L^\infty(\R^n)}^2 \delta^{- n - 2 s} \left[ 2 \sqrt{n - 1} \right]^{n - 1} (B - A)^2,
\end{align*}
where we also used the fact that~$\supp(\varphi)$ is contained in the strip~$\R^{n - 1} \times [A, B]$.

On the other hand, if~$x \in \tRn$ and~$y \in \D_2$, then~$|x'| \le \sqrt{n - 1}/2$ and~$|y'| \ge \sqrt{n - 1}$. Hence,
$$
|x - y| \ge |x' - y'| \ge |y'| - |x'| \ge \frac{|y'|}{2},
$$
and thus
\begin{align*}
\int_{\tRn} \int_{\D_2} \frac{|\tilde{\varphi}(x)| |\tilde{\varphi}(y)|}{|x - y|^{n + 2 s}} \, dx dy & \le 2^{n + 2 s} \| \varphi \|_{L^\infty(\R^n)}^2 (B - A)^2 \int_{\R^{n - 1} \setminus B'_{\sqrt{n - 1}}} \frac{dy'}{|y'|^{n + 2 s}} \\
& \le c_n \| \varphi \|_{L^\infty(\R^n)}^2 (B - A)^2,
\end{align*}
for some dimensional constant~$c_n > 0$. This concludes the proof.
\end{proof}

\section{A remark on separability in~$L^p_\loc$ spaces} \label{sepapp}

We discuss here some separability properties of the subsets of the space~$L^p_\loc(\R^n)$ of locally~$p$-summable functions, for~$1 \le p < +\infty$. While the literature on the standard Lebesgue spaces~$L^p(\R^n)$ is large and exhaustive,~$L^p_\loc(\R^n)$ classes are somehow rarely considered as functional spaces. As we have not been able to find precise references for the few facts about~$L^p_\loc(\R^n)$ that we took advantage of in Proposition~\ref{minminismin}, we provide directly here a proof of such results.

\medskip

First, with the aid of the following proposition, we endow~$L^p_\loc(\R^n)$ with a separable metric made up on the exhaustion of balls~$\bigcup_{k \in \N} \limits B_k$ of~$\R^n$.

\begin{proposition} \label{Lplocsepprop}
Let~$1 \le p < +\infty$ and define
$$
d(u, v) := \sum_{\ell = 1}^{+ \infty} \frac{1}{2^\ell} \frac{\| u - v \|_{L^p(B_\ell)}}{1 + \| u - v \|_{L^p(B_\ell)}},
$$
for any~$u, v \in L^p_\loc(\R^n)$. Then,~$\left( L^p_\loc(\R^n), d \right)$ is a separable metric space.
\end{proposition}
\begin{proof}
It is straightforward to check that~$d$ is a metric. Thus, we only focus on the proof of the separability.

Since~$L^p(\R^n)$ is separable, we may select a sequence~$\{ u_j \}_{j \in \N}$ which is dense in this space. We claim that~$\{ u_j \}$ is dense in~$\left( L^p_\loc(\R^n), d \right)$, too. For a general function~$v \in L^p_\loc(\R^n)$ and any~$k \in \N$, write
$$
\bar{v}^k := \begin{cases}
v & \mbox{in } B_k \\
0 & \mbox{in } \R^n \setminus B_k.
\end{cases}
$$
Thus,~$\bar{v}^k \in L^p(\R^n)$. Fix now~$u \in L^p_\loc(\R^n)$. For any~$k \in \N$, let~$u_{j_k}$ be such that
$$
\| u - u_{j_k} \|_{L^p(B_k)} \le \| \bar{u}^k - u_{j _k} \|_{L^p(\R^n)} \le 2^{- k}.
$$
Of course, such~$u_{j_k}$ exists in view of the density of~$\{ u_j \}$ in~$L^p(\R^n)$. Moreover, we can choose~$\{ j_k \}$ to be increasing in~$k$, so that~$\{ u_{j_k} \}$ is a subsequence of~$\{ u_j \}$. For any~$k$, we then have
\begin{align*}
d(u_{j_k}, u) & = \sum_{\ell = 1}^k \frac{1}{2^\ell} \frac{\| u_{j_k} - u \|_{L^p(B_\ell)}}{1 + \| u_{j_k} - u \|_{L^p(B_\ell)}} + \sum_{\ell = k + 1}^{+\infty} \frac{1}{2^\ell} \frac{\| u_{j_k} - u \|_{L^p(B_\ell)}}{1 + \| u_{j_k} - u \|_{L^p(B_\ell)}} \\
& \le \| u_{j_k} - u \|_{L^p(B_k)} \sum_{\ell = 1}^k \frac{1}{2^\ell} + \sum_{\ell = k + 1}^{+\infty} \frac{1}{2^\ell} \\
& \le \frac{1}{2^{k - 1}},
\end{align*}
and hence~$d(u_{j_k}, u) \rightarrow 0$ as~$k \rightarrow +\infty$. It follows that~$\{ u_j \}$ is dense in~$\left( L^p_\loc(\R^n), d \right)$.
\end{proof}

Now that we have established this property, we can proceed to the kind of separability we are most interested in.

\begin{proposition} \label{aesepprop}
Let~$1 \le p < +\infty$. Then, any subset~$X$ of~$L^p_\loc(\R^n)$ is separable with respect to pointwise a.e. convergence. That is, there exists a sequence~$\{ u_j \}_{j \in \N} \subseteq X$ such that, for any~$u \in X$, a subsequence~$\{ u_{j_k} \}$ of~$\{ u_j \}$ converges to~$u$ a.e. in~$\R^n$.
\end{proposition}
\begin{proof}
First of all, we point out that if~$v_j \rightarrow v$ in~$\left( L^p_\loc(\R^n), d \right)$, then~$v_j$ also converges to~$v$ in~$L^p(B_k)$, for any~$k \in \N$. Indeed,
$$
\frac{1}{2^k} \frac{\| v_j - v \|_{L^p(B_k)}}{1 + \| v_j - v \|_{L^p(B_k)}} \le d(v_j, v) \longrightarrow 0,
$$
as~$j \rightarrow + \infty$ and thence the claim follows by noticing that, given a sequence of non-negative real numbers~$\{ a_j \}_{j \in \N}$ and~$a \in [0, +\infty)$,
$$
a_j \longrightarrow a \quad \mbox{if and only if} \quad \frac{a_j}{1 + a_j} \longrightarrow \frac{a}{1 + a},
$$
as~$j \rightarrow +\infty$.

After this preliminary observation, we can now head to the actual proof of the proposition. Note that, since it is a subset of~$L^p_\loc(\R^n)$,~$X$ is itself a separable metric space with respect to~$d$. This follows by applying Proposition~\ref{Lplocsepprop} and, for instance, Proposition~3.25 of~\cite{B11}. Let then~$\{ u_j \}_{j \in \N} \subseteq X$ be a dense sequence. Fixed an element~$u \in X$, by the initial remark we know that there exists a subsequence~$\{ v_j \}$ of~$\{ u_j \}$ such that~$v_j \rightarrow u$ in~$L^p(B_k)$, for any~$k \in \N$.

We perform a diagonal argument in order to extract a further subsequence~$\{ v_j^* \}$ from~$\{ v_j \}$ which converges to~$u$ a.e. in~$\R^n$.

Since~$\{ v_j \}$ converges to~$u$ in~$L^p(B_1)$, we may select a subsequence~$\{ v_j^1 \}$ from~$\{ v_j \}$ which converges to~$u$ a.e. in~$B_1$. Then,~$\{ v_j^1\}$ still converges to~$u$ in~$L^p(B_2)$, as it is a subsequence of~$\{ v_j \}$, and hence there exists another subsequence~$\{ v_j^2 \}$ of~$\{ v_j^1 \}$ converging to~$u$ a.e. in~$B_2$. We keep extracting nested subsequences and obtain, for any~$k$, a subsequence~$\{ v_j^k \} \subseteq \{ v_j^{k - 1} \}$ converging to~$u$ a.e. in~$B_k$. Set~$v_j^* := v_j^j$ for any~$j \in \N$. This new sequence~$\{ v_j^* \}$ is eventually a subsequence of each of the previous sequences. Thus, it converges to~$u$ a.e. in~$B_k$, for any~$k \in \N$, that is a.e. in~$\R^n$
%
%
\end{proof}

\section*{Acknowledgements}

The authors are indebted to Professor Ovidiu Savin for offering several valuable insights and Professor Hans Triebel for some keen comments on a previous version of the paper.


\end{document}